\documentclass[10pt]{amsart}
\usepackage[margin=1.5in]{geometry}
\usepackage{xcolor}
\newtheorem{theorem}{Theorem}[section]
\newtheorem{lemma}[theorem]{Lemma}
\newtheorem{proposition}{Proposition}[section]

\theoremstyle{definition}
\newtheorem{definition}[theorem]{Definition}

\theoremstyle{remark}
\newtheorem{remark}[theorem]{Remark}

\numberwithin{equation}{section}

\begin{document}
\title[Stationary quantum BGK model]{Stationary quantum BGK model for bosons and fermions in a bounded interval}

\author{Gi-Chan Bae}
\address{Department of mathematics, Sungkyunkwan University, Suwon 440-746, Republic of Korea }
\email{gcbae02@skku.edu}

\author{Seok-Bae Yun}
\address{Department of mathematics, Sungkyunkwan University, Suwon 440-746, Republic of Korea }
\email{sbyun01@skku.edu}

%    General info
%\subjclass[2010]{76P05,46N55,82C40,35F25   }

%\date{January 1, 2001 and, in revised form, June 22, 2001.}


\begin{abstract}
In this paper, we consider the existence problem for a stationary relaxational models of the quantum Boltzmann equation. 
More precisely, we establish the existence of mild solution to the fermionic or bosonic quantum BGK model in a slab with inflow boundary data.
Unlike the classical case, it is necessary to verify that the quantum local equilibrium state is well-defined, and the transition from the non-condensed state to
the condensated state (Bosons), or from the non-saturated state to the saturated state (Fermions) does not arise in our solution space.
%In the case of Boson, we consider the non-condensation regime, and we consider the non-saturated regime for the fermions.
%A main difficulty which does not appear in classical BGK model is that we must control some porportion of macroscopic quantities of the solution to maintain the differentiability of quantum equilibrium. We verify the continuity property of quantum equilibrium with \textcolor{red}{intricate} nonlinear relations on there parameters. The mild solution is established under rarefied gas and fixed boundary condition.% which is not much concentrated in $p=0$.
\end{abstract}

\maketitle

\section{Introduction}
The stationary quantum BGK model \cite{Jungel-transport,MRS,MY,N,Nt,RS,WMZ,YMCL,YYDHLZZ} in a bounded interval reads
\begin{align}\label{QBGK}
p_1\frac{\partial f}{\partial x}&=\frac{N}{\tau}(\mathcal{K}(f)-f),
\end{align}
subject to boundary conditions:
\begin{align}
f(0,p)=f_L(p) \quad \textit{for} \quad p_1>0, \qquad
f(1,p)=f_R(p) \quad \textit{for} \quad p_1<0.
\end{align}
The momentum distribution function $f(x,p)$ depends on the position $x\in[0,1]$ and the momentum $p\in\mathbb{R}^3$. The Knudsen number $\tau>0$ measures how rarefied the gas system is, and is defined by the ratio between the characteristic length and mean free path.
Throughout this paper, $\mathcal{K}$ denotes the local equilibrium of the system. For bosonic case, it represents the Bose-Einstein distribution without condensation, and
in the fermionic case, it represents the non-saturated Fermi-Dirac distribution, which will be defined below.
To present the exact form of $\mathcal{K}$, we first define the macroscopic mass, momentum and energy:
\begin{align}\label{NPE}
\begin{split}
N_f(x)=\int_{\mathbb{R}^3}f(x,p)dp,\quad P_f(x)=\int_{\mathbb{R}^3}f(x,p)pdp, \quad
E_f(x)=\int_{\mathbb{R}^3}f(x,p)|p|^2dp.
\end{split}
\end{align}
We then introduce the equilibrium parameter $a$ and $c$ defined by ( $+$ sign is for fermion and $-$ sign is for boson, see \cite{BaeY,N}):
\begin{align}\label{a,c1}
\begin{split}
\frac{N_f(x)}{\left(E_f(x)-\frac{|P_f(x)|^2}{N_f(x)}\right)^{\frac{3}{5}}}= \frac{\displaystyle\int_{\mathbb{R}^3}\frac{1}{e^{|p|^2+c(x)}\pm1}dp}{\displaystyle\left(\int_{\mathbb{R}^3}\frac{|p|^2}{e^{|p|^2+c(x)}\pm1}dp\right)^{\frac{3}{5}}},
%c(x)&=\beta^{-1}\left(\frac{N(x)}{\left(E(x)-\frac{P(x)^2}{N(x)}\right)^{\frac{3}{5}}}\right),	\cr
\end{split}
\end{align}
and
\begin{align}\label{a,c2}
	\displaystyle a(x)&=\left({\int_{\mathbb{R}^3}\frac{1}{e^{|p|^2+c(x)}\pm1}dp}\right)^\frac{2}{3}N(x)^{-\frac{2}{3}}.
\end{align}
Note that $c$ is determined implicitly. For the later convenience, we define
\begin{align}\label{beta}
	\beta_{\mathcal{K}}(c)= \frac{\int_{\mathbb{R}^3}\frac{1}{e^{|p|^2+c}\pm1}dp}{\left(\int_{\mathbb{R}^3}\frac{|p|^2}{e^{|p|^2+c}\pm1}dp\right)^{\frac{3}{5}}}.
\end{align}
The relations \eqref{a,c1} and \eqref{a,c2} arise from the requirement that $\mathcal{F}$, $\mathcal{B}$ must share the same mass, momentum and energy with $f$ (See \cite{BaeY}):
%In this chapter we consider how \eqref{a,c1} and \eqref{a,c2} are derived and why condition \eqref{inibeta} is necessary. \eqref{a,c1} and \eqref{a,c2} arise from the requirement that the local equilibrium and the momentum distribution function share the same mass, momentum and energy \cite{BaeY,N}:
%\begin{align*}
%\int_{\mathbb{R}^3}(\mathcal{K}(f)-f)\left(\begin{array}{c}1 \cr p \cr |p|^2\end{array}\right) dp = 0.
%\end{align*}
Now we are ready to define the local quantum equilibriums. \cite{BaeY,EMV,LuB1,N,Za} \newline
%%%%%%%%%%%%%%%%%%%%%%%%%%%%%%%%%%%%%%%%%%%%%%%%%%%%%%%%%%%%%%%%%%%%%%%%%%%%%%%%%%%%%%%%%%%%%%%%%%%%%%%%%%%%%%%%%%%%%%%%%%%%%%%%%%%%%%%%%%%%%%%%%%%%%%%%%%%%%%%%%%%%%%%%%%%%%%%%%%%%
\noindent $\bullet$ ~ {\bf Bose-Einstein distribution:} The local equilibrium for bosons is defined as follows:
\begin{align}\label{BoseEin}
%\displaystyle\mathcal{B}(f)(x,p,t)=
\begin{cases}
\mathcal{B}_1(f)=\left(e^{a(x)\big|p-\frac{P(x)}{N(x)}\big|^2+c(x)}-1\right)^{-1}, ~ &\textit{if} \quad \frac{N}{\left(E-\frac{|P|^2}{N}\right)^{3/5}}\leq \beta_{\mathcal{B}}(0),\cr
\mathcal{B}_2(f)=\left(e^{a(x)\big|p-\frac{P(x)}{N(x)}\big|^2}-1\right)^{-1}+k(x)\delta_{p=\frac{P(x)}{N(x)}},
\quad &\textit{otherwise},
\end{cases}
\end{align}
where
\begin{align*}
k(x)=N(x)-\beta_{\mathcal{B}}(0)\left(E(x)-\frac{P(x)^2}{N(x)}\right)^{\frac{3}{5}}.
\end{align*}
The dirac delta function corresponds to Bose-Einstein condensation. $\mathcal{B}_1$ corresponds to the non-condensation case, while $\mathcal{B}_2$ is referred as the condensation case. \newline
\noindent $\bullet$ ~ {\bf Fermi-Dirac distribution:} The local equilibrium for fermions is defined as follows:
\begin{align}\label{FD}
%\displaystyle\mathcal{F}(f)(x,p,t)
\begin{cases}
\mathcal{F}_1(f)=\left(e^{a(x)\big|p-\frac{P(x)}{N(x)}\big|^2+c(x)}+1\right)^{-1}, ~ &\textit{if} \quad \frac{N}{\left(E-\frac{|P|^2}{N}\right)^{3/5}}< \beta(-\infty),	\cr
\mathcal{F}_2(f)=\chi_{\big|p-\frac{P(x)}{N(x)}\big|\leq \left(\frac{3N(x)}{4\pi}\right)^\frac{1}{3}}, \quad &\textit{otherwise} ,\end{cases}
\end{align}
where $\chi_{A}$ denotes the characteristic function on $A$, and the second case of $\mathcal{F}_2$ is called the saturated Fermi-Dirac distribution.\newline

Throughout this paper, we will use $\mathcal{B}(f)$ to denote the Bose-Enstein distribution without condensation $\mathcal{B}_1(f)$, while  $\mathcal{F}(f)$
is used to denote the non-saturated Fermi-Dirac distribution $\mathcal{F}_1(f)$. Also, $\mathcal{K}(f)$ denotes either $\mathcal{B}(f)$ or $\mathcal{F}(f)$.
%For the derivation of quantum equilibrium, we cite \cite{EMV}. If $\beta$ is strictly monotone function, then we can determine $c$ uniquely. The monotonicity of $\beta_{\mathcal{B}}$ on $c\geq0$, is proved in \cite{LuB1}, and monotonicity of $\beta_{\mathcal{F}}$ on $c\geq-\ln3$ is proved in \cite{BaeY}, but for the convenience of the readers, we present the whole proof in Appendix A for the case of boson.
\newline
%The relaxation operator satisfies
%\begin{align}\label{conserv}
%\int_{\mathbb{R}^3}(\mathcal{K}(f)-f) \left( \begin{array}{c} 1 \cr p \cr %|p|^2 \end{array}\right) dp = 0.
%\end{align}
%\begin{align*}
%\int_{\mathbb{R}^3}\mathcal{B}(f)(x,p,t)\left(\begin{array}{c}1\cr p\cr %|p|^2\end{array}\right)dp
%=\int_{\mathbb{R}^3}f(x,p,t)\left(\begin{array}{c}1\cr p\cr %|p|^2\end{array}\right)dp.
%\end{align*}
%This cancellation property implies the conservation of the following quantities on $x\in[0,1]$.
%\begin{align*}
%\int_{\mathbb{R}^3}fp_1dp, \quad \int_{\mathbb{R}^3}fp_1pdp, \quad %\int_{\mathbb{R}^3}fp_1|p|^2dp.
%\end{align*}
%We also have the following celebrated H-theorem :
%\begin{align*}
%\frac{d}{dt}\int_{\mathbb{T}^3\times\mathbb{R}^3}f\ln f+(1- f)\ln (1 - f) dxdp \leq 0,
%\end{align*
%for fermions,
%\begin{align*}
%\frac{d}{dt}\int_{\mathbb{T}^3\times\mathbb{R}^3}f\ln f-(1+ f)\ln (1 + f) dxdp \leq 0,
%\end{align*}
%for bosons.
\subsection{Brief history}
The slab problem corresponds to the situation where there is a gas flow between two parallel gas-emitting plates of infinite size. This  arise often in science and engineering,
and attracted the interest of many researchers.
%Stationary problem for the classical boltzmann and BGK model was considered in various type of boundary conditions.
In the case of the Boltzmann equation, the first mathematical study can be traced back to \cite{AC}, where the existence of a measure valued solution were investigated.
In the framework of weak solutions, Arkeryd and Nouri considered the existence of $L^1$ solution for the inflow boundary conditions in \cite{AN00,AN98} and
for the diffusive reflection conditions in \cite{AN00}. These results were extended to gas mixture problem by Brull \cite{Br1,Br2}.
%For the given boundary condition Arkeyd and nouri considered the existence of $L^1$ solution in \cite{AN00,AN98}. the same authors obtained weak $L^1$ compactness and weak $L^1$ solution for Boltzmann equation in .
Gomeshi studied the existence of unique mild solutions under the condition that the Knudsen number is sufficiently large in \cite{G}.
For the related 3d problem near equilibrium, see \cite{Kim1,Kim2}.

In the case of BGK type model, Ukai studied stationary Boltzmann BGK model in slab for fixed large boundary data in \cite{Ukai} using a Schauder type fixed
point theorem.
Nouri \cite{N} established the existence of weak solutions for the stationary quantum BGK model with a discretized condensation term in a slab.
%\textcolor{red}{The quantum transport thoery in the inflow boundary condition with relaxation time approximation was considerd by the authors in \cite{LTZ}.}
Bang and Yun obtained the existence and uniqueness of mild solutions for the ES-BGK model under the assumption that gas is sufficiently rarefied and inflow datas are not concentrated on $p=0$ in \cite{BY}. %In the case of polyatomic particle, Park and Yun obtained existence theory in \cite{PY}.
%the entropy minimisation problem was proved in \cite{EMV} by Escobedo, Mischler and Valle, and Lu obtained well-posedness of proper parameter satisfying conservation laws in \cite{LuB1}.

The mathematical reserach for the quantum relaxation model has just started, and the literature remains extremely limited. The first mathmatical study was carried out by Nouri as mentioned above. In \cite{Brauk1,Brauk2}, Braukhoff obtained analytic solutions of quantum BGK type model arising in the study of ultracold fermionic clouds.
The global existence and asymptotic behavior of fermionic quantum BGK model near a global Fermi-Dirac distribution were studied by the authors in \cite{BaeY}.
Presently, authors are not aware of any further analytical results on the quantum BGK models.  We refer to \cite{FHJ,JP,Jin,HJW,MY,SY,WMZ,YH} for numerical studies on the quantum BGK model.
%In this paper, we consider the stationary quantum BGK model for quantum particle in 1-d slab. Existence and uniqueness are considered under the assumption that gas is sufficiently rarefied and inflow boundary data is not concentrated at $p_1=0$.

Quantum Boltzmann equation, on the other hand, has seen more progress.
We refer to \cite{BE,EMVa,EMV1,EV1,LL,LuB1,LuB7,LuB6,LuF1,NT,ST} for homogeneous problem, and \cite{AT,AN1,AN3,D} for inhomogeneous problems.

\subsection{Notations}
We define notations and norms that are frequently used throughout this paper.
\begin{itemize}
\item
Throughout this paper, we fix $\mathcal{B}(f)=\mathcal{B}_1(f)$ and $\mathcal{F}(f)=\mathcal{F}_1(f)$. Also, $\mathcal{K}(f)$ denotes either $\mathcal{B}(f)$ or $\mathcal{F}(f)$.
\item Every constants $C$ are defined generically. We also use $C_{a,b,\cdots}$ when it is necessary to explicitly show the dependence on $a,b,\cdots$. Especially, we denote $C_{l,u}$ when the constant depends only on the constants defined in (\ref{iniconst}) and $k$.
\item When there's no risk of confusion, we suppress the dependence of the macroscopic fields on $f$, and denote $N$, $P$, $E$
instead of $N_f$, $P_f$ and $E_f$.
\item We define our weighted $L^1$ norm and weighted $L^{\infty}$ as follows:
\begin{align*}
\sup_{x}||f||_{L_2^1}&=\sup_x\left\{\int_{\mathbb{R}^3}|f(x,p)|(1+|p|^2)dp\right\},	\cr
||f||_{L_2^{\infty}}&=\sup_{x,p}|f(x,p)|(1+|p|^2).
\end{align*}
\item
We use the following notation (See Ch $3$):
\begin{align*}
\beta_{\mathcal{B}}^{-1}=\big(\beta_{\mathcal{B}}\big|_{[0,\infty)}\big)^{-1}, \quad \beta_{\mathcal{F}}^{-1}=\big(\beta_{\mathcal{F}}\big|_{(-\ln3,\infty)}\big)^{-1}.
\end{align*}
\end{itemize}
This paper is organized as follows. In Section 2 we present the main result and give an example of boundary data satisfying the assumption of main theorem.
Section 3 is devoted to the fixed point setup of the problem. We define the solution space and prove that the equilibrium is well defined in this space. Some useful estimates are also introduced in this section. In Section 4, we establish that the solution operator maps the solution space into itself.  We prove the main theorem in the final Section 5 by showing that the solution operator is a contraction mapping.
\section{Main result}
In this section we present our main results. For brevity we denote
\begin{align}\label{LR}
f_{LR}(p)=f_L(p)1_{p_1>0}+f_R(p)1_{p_1<0},
\end{align}
and define the following quantities:
\begin{align}\label{iniconst}
\begin{split}
&a_u=2\int_{\mathbb{R}^3}f_{LR}dp, \quad a_l=\int_{\mathbb{R}^3}e^{-\frac{a_u}{\tau|p_1|}}f_{LR}dp,\quad a_s=\int_{\mathbb{R}^3}\frac{1}{|p_1|}f_{LR}dp,	\cr
&c_u=2\int_{\mathbb{R}^3}f_{LR}|p|^2dp,\quad c_l=\int_{\mathbb{R}^3}e^{-\frac{a_u}{\tau|p_1|}}f_{LR}|p|^2dp,\quad c_s=\int_{\mathbb{R}^3}\frac{1}{|p_1|}f_{LR}|p|^2dp,
\end{split}
\end{align}
and
\begin{align}\label{k}
k=\left(\int_{p_1>0}e^{-\frac{a_u}{\tau|p_1|}}f_L(p)|p_1|dp\right)
\left(\int_{p_1<0}e^{-\frac{a_u}{\tau|p_1|}}f_R(p)|p_1|dp\right).
\end{align}
\begin{definition}\label{mildsol}
We say that $f\in L^1_2([0,1]\times\mathbb{R}^3_p)$ is a mild solution of  (\ref{QBGK}) if $f$ satisfies
\begin{align*}
\begin{split}
f(x,p)&=e^{-\frac{1}{\tau|p_1|}\int_0^xN_f(y)dy}f_L(p)	\cr
&+\frac{1}{\tau|p_1|}\int_0^xe^{-\frac{1}{\tau|p_1|}\int_y^xN_f(z)dz}N_f(y)\mathcal{K}(f)dy \quad \textit{if} \quad p_1>0,
\end{split}
\end{align*}
and
\begin{align*}
\begin{split}
f(x,p)&=e^{-\frac{1}{\tau|p_1|}\int_x^1N_f(y)dy}f_R(p)	\cr
&+\frac{1}{\tau|p_1|}\int_x^1e^{-\frac{1}{\tau|p_1|}\int_x^yN_f(z)dz}N_f(y)\mathcal{K}(f)dy \quad \textit{if} \quad p_1<0.
\end{split}
\end{align*}
\end{definition}
Now we state our main results.
\begin{theorem}\label{mainthm}
Assume $f_L$ and $f_R$ satisfy the following conditions:
\begin{enumerate}
\item Boundary data are non-negative:
\begin{align*}
f_{LR}\geq0,
\end{align*}
\item Boundary data satisfy the following integrability conditions:
\begin{align*}
f_{LR}, \quad \frac{1}{|p_1|}f_{LR}\in L_2^1,
\end{align*}
\item Contributions of the inflow from the boundary in $p_2$ and $p_3$ directions are negligible:
\begin{align*}
\int_{\mathbb{R}^2}f_Lp_idp_2dp_3=\int_{\mathbb{R}^2}f_Rp_idp_2dp_3=0. \quad (i=2,3)
\end{align*}
\end{enumerate}
We assume further that
\begin{align}\label{inibeta}
\frac{{a_u}^{\frac{8}{5}}}{k^{\frac{3}{5}}}<\beta_{\mathcal{B}}(0)\quad(\mbox{Boson}),~\quad \frac{{a_u}^{\frac{8}{5}}}{k^{\frac{3}{5}}}<\beta_{\mathcal{F}}(-\ln 3)\quad (\mbox{Fermion}).
\end{align}
Then for sufficiently large $\tau$, there exists a unique non-negative mild solution $f$ of (\ref{QBGK}) satisfying
\begin{align}\label{upper lower}
a_l\leq N(x)\leq a_u, \quad c_l\leq E(x)\leq c_u,
\end{align}
and
\begin{align}\label{lower k}
%\left(\int_{\mathbb{R}^3}fdp\right)\left(\int_{\mathbb{R}^3}f|p|^2dp\right)-\left(\int_{\mathbb{R}^3}fp_1dp\right)^2 \geq k.
E(x)N(x)-|P(x)|^2\geq k.
\end{align}
\end{theorem}
\begin{remark} \label{rmk}
%(1) Since there is no macroscopic flow in $y$ and $z$ axes in the $1$-dimension, the assumption of Theorem \ref{mainthm} (3) is physically plausible.
%\\
%(2) The macroscopic moments $P$ is well defined as following sense:
%\begin{align*}
%\bigg|\int_{\mathbb{R}^3}f(x,p)pdp\bigg|\leq a_u+c_u.
%\end{align*}
%(3) Assumption of Theorem \ref{mainthm} (2) is critically used when we prove the upper bounds of macroscopic mass and energy of mild solution. And this assumption is only technical. It is open question whether this existence theory will be maintained if this assumption is relieved or removed.
(1) The meaning of assumption \eqref{inibeta} will be considered in Chapter 3.
(2) Note that in (\ref{inibeta}), the fermion case is restricted to $\beta(-\ln3)$.
 This is because we don't know yet
whether $\beta(c)$ for fermion is a strictly monotone decreasing function in the whole range, even though the numerics indicate in that way. This is left as a future preject.
(3) Extending this result to include the condensated state (Boson) and the saturated state (Fermion) will be  interesting, and is left for the future.
\end{remark}
Before we move on the the proof of the theorem,  we present a simple example of boundary data which satisfies the assumption of Theorem \ref{mainthm} (1), (2), (3) and (\ref{inibeta}) for bosons. Example for fermionic particles can be constructed similarly. We define
\begin{align*}
f_L(p)=C_L1_{r_1\leq p_1 \leq r_2}e^{-\frac{|p_2|^2}{2}-\frac{|p_3|^2}{2}}, \quad f_R(p)=C_R1_{-r_2\leq p_1 \leq -r_1}e^{-\frac{|p_2|^2}{2}-\frac{|p_3|^2}{2}},
\end{align*}
for some $C_L,C_R>0$ and $r_1,r_2>0$ to be determined soon. Since it can be readily checked that they satisfy the conditions $(1)$, $(2)$, $(3)$ of Theorem \ref{mainthm},
we check the condition (\ref{inibeta}) only. We first compute $a_u$ as
\begin{align*}
a_u &= 2\int_{\mathbb{R}^3}f_{LR} dp \cr
&=2\left(C_L\int_{r_1}^{r_2}1dp_1\left(\int_{-\infty}^{\infty}e^{-\frac{|p_2|^2}{2}}dp_2\right)^2+C_R\int_{-r_2}^{-r_1}1dp_1\left(\int_{-\infty}^{\infty}e^{-\frac{|p_2|^2}{2}}dp_2\right)^2\right) \cr
&=4\pi (C_L+C_R)(r_2-r_1).
\end{align*}
We then compute
\begin{align*}
\int_{p_1>0}e^{-\frac{a_u}{\tau |p_1|}}f_L(p)p_1dp &\geq e^{-\frac{4\pi (C_L+C_R)(r_2-r_1)}{\tau r_1}}\int_{\mathbb{R}^3}f_L(p)p_1dp \cr
&=C_Le^{-\frac{4\pi (C_L+C_R)(r_2-r_1)}{\tau r_1}}\int_{r_1}^{r_2}p_1dp_1 \left(\int_{-\infty}^{\infty}e^{-\frac{|p_2|^2}{2}}dp_2\right)^2	\cr
&= \pi C_L e^{-\frac{4\pi (C_L+C_R)(r_2-r_1)}{\tau r_1}} (r_2^2-r_1^2),
\end{align*}
and, similarly,
\begin{align*}
\int_{p_1<0}e^{-\frac{a_u}{\tau|p_1|}}f_R(p)|p_1|dp &\geq \pi C_R e^{-\frac{4\pi (C_L+C_R)(r_2-r_1)}{\tau r_1}} (r_2^2-r_1^2),
\end{align*}
to get
\begin{align*}
k &= \left(\int_{p_1>0}e^{-\frac{a_u}{\tau|p_1|}}f_L(p)|p_1|dp\right)
\left(\int_{p_1<0}e^{-\frac{a_u}{\tau|p_1|}}f_R(p)|p_1|dp\right) \cr
&= \pi^2 C_LC_Re^{-\frac{8 \pi (C_L+C_R)(r_2-r_1)}{\tau r_1}} (r_2^2-r_1^2)^2.
\end{align*}
Hence we derive
\begin{align*}
\frac{{a_u}^{\frac{8}{5}}}{k^{\frac{3}{5}}} &= \frac{\left(4\pi (C_L+C_R)(r_2-r_1)\right)^{\frac{8}{5}}}{\left(\pi^2 C_LC_Re^{-\frac{8 \pi (C_L+C_R)(r_2-r_1)}{\tau r_1}} (r_2^2-r_1^2)^2\right)^{\frac{3}{5}}}	\cr
&=4^{\frac{8}{5}}\pi^{\frac{2}{5}}\frac{(C_L+C_R)^\frac{8}{5}}{(C_LC_R)^{\frac{3}{5}}}\frac{(r_2-r_1)^{\frac{2}{5}}}{(r_2+r_1)^{\frac{6}{5}}}e^{\frac{24 \pi (C_L+C_R)(r_2-r_1)}{5 \tau r_1}}.
\end{align*}
This shows that a proper choice of $C_L$, $C_R$ and $r_1$, $r_2$ gives the desired condition.
%If we choose $C_L=C_R=1$, $r_1=5$ and $r_2=5.1$ then the condition (\ref{inibeta}) holds by a numeric calculation of $\beta_{\mathcal{B}}(0)$ for any $\tau\geq1$, that is
%\begin{align*}
%\frac{N}{\left(E-\frac{|P|^2}{N}\right)^{\frac{3}{5}}} = \frac{N^{\frac{8}{5}}}{\left(EN-|P|^2\right)^{\frac{3}{5}}}
%\leq\frac{{a_u}^{\frac{8}{5}}}{k^{\frac{3}{5}}}\leq \beta_{\mathcal{B}}(0).
%\end{align*}
%Therefore such boundary function satisfies all the assumptions of main theorem.
%\\
%(6) $a_l$ and $k$ is defined in $\tau$.

\section{Fixed point set-up}
%In this section, we define the solution space and establish some useful lemmas. %operator $\Phi$ to use the Banach fixed point theorem. %And we make several important estimates.
We define the solution space by
\begin{align*}
\Lambda = \left\{f\in L_2^1([0,1]\times\mathbb{R}^3_p)|~f ~ \textit{satisfies}~ (\mathcal{A},\mathcal{B},\mathcal{C})\right\},
\end{align*}
endowed with the metric $d(f,g)=\sup_{x\in[0,1]}||f-g||_{L^1_2}$.
\begin{itemize}
\item ($\mathcal{A}$) $f$ is non-negative:
\begin{align*}
f(x,p)\geq 0 \quad \textit{for} \quad x,p \in [0,1] \times \mathbb{R}^3.
\end{align*}
\item ($\mathcal{B}$) Mass and energy satisfy
\begin{align*}
a_l \leq \int_{\mathbb{R}^3}f(x,p)dp \leq a_u, \quad c_l \leq \int_{\mathbb{R}^3}f(x,p)|p|^2dp \leq c_u.
\end{align*}
\item ($\mathcal{C}$) $f$ satisfies
\begin{align*}
\left(\int_{\mathbb{R}^3}f(x,p)dp\right)\left(\int_{\mathbb{R}^3}f(x,p)|p|^2dp\right)-\left|\int_{\mathbb{R}^3}f(x,p)pdp\right|^2 \geq k.
\end{align*}
\end{itemize}
%%%%%%%%%%%%%%%%%%%%%%%%%%%%%%%%%%%%%%%%%%%%%%%%%%%%%%%%%%%%%%%%%%%%%%%%%%%%%%%%%%%%%%%%%%%%%%%%%%%%%%%%%%%%%%%%%%%%%%%%%%%%%%%%%%%%%%%%%%%%%%%%%%%%%%%%%%%%%%%%%%%%%%%%%%%%%%%%%%%%%%%%5
\subsection{determination of $a$,$b$ and $c$}
We first verify that for any distribution function $f$ that lies in $\Lambda$, the nonlinear relations \eqref{a,c1} and $\eqref{a,c2}$ admit a unique set of solution $a$ and $c$, so that the local equilibrium $\mathcal{K}(f)$ is well defined. It is clear that $a$ is uniquely determined by \eqref{a,c2}  once the unique existence of $c$ is determined from (\ref{a,c1}).
%Therefore, we only need to show that when $f$ is an element of $\Lambda$, $c$ is uniquely determined by $\eqref{a,c1}$.
Note that, in view of the definition of (\ref{beta}), the nonlinear relation (\ref{a,c1}) is rewritten by
\begin{align}\label{by}
\beta_{\mathcal{F}}(c)=\frac{N(x)}{\left(E(x)-\frac{|P(x)|^2}{N(x)}\right)^{\frac{3}{5}}}.
\end{align}
Therefore, it is sufficient to show that $\beta_{\mathcal{F}}$ is a monotone function, and r.h.s of (\ref{by}) lies in the range of $\beta_{\mathcal{F}}$.
For this we recall the following lemma:
\begin{lemma}\label{monotone}\emph{\cite{BaeY,LuB1}}
The function $\beta_{\mathcal{B}}$ and $\beta_{\mathcal{F}}$ defined in (\ref{beta}) satisfy the following properties.
\begin{enumerate}
\item $\beta_{\mathcal{B}}$ is strictly decreasing on $[0,\infty)$ and its range is $(0,\beta(0)]$.
\item$\beta_{\mathcal{F}}$ is strictly decreasing on $(-\ln3,\infty)$ and its range is $(0,\beta(-\ln 3))$.
\end{enumerate}
\end{lemma}
\begin{proof}
Proof for (1) can be founded in \cite{LuB1}, and the proof for (2) can be founded in \cite{BaeY}. %\textcolor{red}{But for the convenience of the readers, we represent the proof only for the boson case in Appendix A.}
\end{proof}
%The condition \eqref{inibeta} provides a sufficient condition for the local Bose-Einstein distribution $\mathcal{B}$ and the local Fermi-Dirac distribution $\mathcal{F}$ to be well defined, in the sense that the equilibrium parameter $a$ and $c$ are uniquely determined by the relation (\ref{a,c1}), (\ref{a,c2}) under the condition \eqref{inibeta}.
%The first condition in \eqref{inibeta} guarantees the absence of Bose-Einstein condensation.
\begin{lemma}\label{determination}
Assume $f\in\Lambda$. Then $a$ and $c$ are uniquely determined from (\ref{a,c1}) and (\ref{a,c2}), and $\mathcal{K}(f)$ is well-defined.
Moreover, $\mathcal{K}(f)$ is not condensated (Bosonic case) nor saturated (Fermionic case).
That is, no transition from $\mathcal{B}_1(f)$ to $\mathcal{B}_2(f)$, or $\mathcal{F}_1(f)$ to $\mathcal{F}_2(f)$ occurs.
\end{lemma}
\begin{proof}
$\bullet$~(Boson): We note from  \eqref{inibeta}$_1$, \eqref{upper lower} and \eqref{lower k} that
\begin{align*}
%\beta_{\mathcal{B}}(c)=\frac{\int_{\mathbb{R}^3}\frac{1}{e^{|p|^2+c}-1}dp}{\left(\int_{\mathbb{R}^3}\frac{|p|^2}{e^{|p|^2+c}-1}dp\right)^\frac{3}{5}}=
0<\frac{N}{\left(E-\frac{|P|^2}{N}\right)^{\frac{3}{5}}} =
\frac{N^{\frac{8}{5}}}{(EN-|P|^2)^{\frac{3}{5}}} \leq \frac{{a_u}^{\frac{8}{5}}}{k^{\frac{3}{5}}}<\beta_{\mathcal{B}}(0).
\end{align*}
Therefore, in view of Lemma \ref{monotone}, the interval $\big(0, N^{\frac{8}{5}}/(EN-|P|^2)^{\frac{3}{5}}\big]$ lies in the range of $\beta_{\mathcal{B}}$, and we can fix a unique $c$ satisfying (\ref{by}) by the monotonicity of $\beta_{\mathcal{B}}$ obtained in Lemma \ref{monotone}, which in turn leads to the determination of $a$ by (\ref{a,c2}).
Note also from (\ref{BoseEin}) that this guarantees that the condensation does not arise if $f\in\Lambda$.
In conclusion, $\mathcal{B}(f)$ is well-defined for $f\in \Lambda$.\\
\noindent$\bullet$~(Fermion): Similarly, combining second condition of \eqref{inibeta}$_2$ with \eqref{upper lower} and \eqref{lower k} yields
\begin{align*}
\frac{N}{\left(E-\frac{|P|^2}{N}\right)^{\frac{3}{5}}} =\frac{N^{\frac{8}{5}}}{(EN-|P|^2)^{\frac{3}{5}}} \leq \frac{{a_u}^{\frac{8}{5}}}{k^{\frac{3}{5}}}<\beta_{\mathcal{F}}(-\ln 3),
\end{align*}
for fermion case. Therefore, by the exactly same argument, we can conclude that $a$ and $c$ are uniquely determined for $f\in\Lambda$, and the transition from the non-saturated state $\mathcal{F}_1(f)$ to the saturated state $\mathcal{F}_2(f)$ does not happen.
\end{proof}
In view of this consideration, we can uniquely determine $c$ satisfying \eqref{a,c1}. For brevity, we slightly abuse the notation to denote as
\begin{align}\label{beta inv}
\beta_{\mathcal{B}}^{-1}=\big(\beta_{\mathcal{B}}\big|_{[0,\infty)}\big)^{-1}, \quad \beta_{\mathcal{F}}^{-1}=\big(\beta_{\mathcal{F}}\big|_{(-\ln3,\infty)}\big)^{-1},
\end{align}
and $\beta_{\mathcal{K}}^{-1}$ will denote
\[
\beta_{\mathcal{K}}^{-1}=\beta_{\mathcal{B}}^{-1} ~~(\mbox{Boson}) \quad \mbox{and}\quad \beta_{\mathcal{K}}^{-1}=\beta_{\mathcal{F}}^{-1}~~(\mbox{Fermion}).
\]
%%%%%%%%%%%%%%%%%%%%%%%%%%%%%%%%%%%%%%%%%%%%%%%%%%%%%%%%%%%%%%%%%%%%%%%%%%%%%%%%%%%%%%%%%%%%%%%%%%%%%%%%%%%%%%%%%%%%%%%%%%%%%%%%%%%%%%%%%%%%%%%%%%%%%%%%%%%%%%%%%%%%%%%%%%%%%%%%55
We first consider the range of $a$ and $c$ when they are constructed from an element of $\Lambda$.
%We first show that the equilibrium parameters, and hence the quantum local equilibriums are well-defined for functions in $\Lambda$.
%In this setting we show that $\Phi(\Lambda) \subset \Lambda$ in Section $4$ and in Section 5, we prove that $\Phi$ is contraction map.

\begin{lemma}\label{bddac} Let $f\in\Lambda$, and the boundary data $f_{LR}$ satisfy \eqref{inibeta}.
Define $a_*$, $a^*$, $c_*$, $c^*$ by
\[
c_* = \beta_{\mathcal{K}}^{-1}\left(\frac{a_u^{\frac{8}{5}}}{k^{\frac{3}{5}}}\right),\quad  c^*=\beta_{\mathcal{K}}^{-1}\left(\frac{a_l^{\frac{8}{5}}}{\left(a_uc_u\right)^{\frac{3}{5}}}\right),
\]
and
\begin{align*}
a_*=\left({\int_{\mathbb{R}^3}\frac{1}{e^{|p|^2+c^*}\pm1}dp}\right)^\frac{2}{3}a_u^{-\frac{2}{3}}, \quad a^*=\left({\int_{\mathbb{R}^3}\frac{1}{e^{|p|^2+c_*}\pm1}dp}\right)^\frac{2}{3}a_l^{-\frac{2}{3}}.
\end{align*}
Then, the equilibrium parameter $a$ and $c$ satisfy
\[
 0\leq a_* \leq a \leq a^*.
\]
and
\begin{align*}
 -\ln3 \leq c_* \leq c \leq c^* \quad (\mbox{Fermion})\quad\mbox{and} \quad 0 \leq c_* \leq c \leq c^*\quad (\mbox{Boson}).
\end{align*}
In the case of fermion, we note that $-\ln3  \leq c_*$.
\end{lemma}
\begin{proof}
(1) {\bf Estimates for $c$}:
From \eqref{a,c1} and \eqref{beta}, we have
\begin{align*}
\beta_{\mathcal{K}}(c) =\frac{N}{\left(E-\frac{|P|^2}{N}\right)^{3/5}}= \frac{N^{\frac{8}{5}}}{\left(EN-|P|^2\right)^{\frac{3}{5}}}.
\end{align*}
Since $f\in\Lambda$ we have $a_l\leq N\leq a_u$ , $E\leq c_u$ and $EN-|P|^2 \geq k$, so that
\begin{align*}
\frac{a_l^{\frac{8}{5}}}{\left(a_uc_u\right)^{\frac{3}{5}}} \leq \beta_{\mathcal{K}}(c) \leq  \frac{a_u^{\frac{8}{5}}}{k^{\frac{3}{5}}}.
\end{align*}
Now, since Lemma \ref{monotone} implies that $\beta_{\mathcal{K}}^{-1}$ is strictly decreasing, and
the closed interval $\big[{a_l^{{8}/{5}}}/{\left(a_uc_u\right)^{{3}/{5}}},{a_u^{{8}/{5}}}/{k^{{3}/{5}}}\big]$ lies in the range of $\beta_{\mathcal{K}}(c)$, we have
\begin{align*}
0 \leq \beta_{\mathcal{K}}^{-1}\left(\frac{a_u^{\frac{8}{5}}}{k^{\frac{3}{5}}}\right)  \leq c \leq  \beta_{\mathcal{K}}^{-1}\left(\frac{a_l^{\frac{8}{5}}}{\left(a_uc_u\right)^{\frac{3}{5}}}\right),
\end{align*}
%for boson case, and
%\begin{align*}-\ln 3 \leq \beta_{\mathcal{F}}^{-1}\left(\frac{a_u^{\frac{8}{5}}}{k^{\frac{3}{5}}}\right)  \leq c \leq  \beta_{\mathcal{F}}^{-1}\left(\frac{a_l^{\frac{8}{5}}}{\left(a_uc_u\right)^{\frac{3}{5}}}\right),\end{align*}for fermion case.
%Setting \[c_* = \beta_{\mathcal{K}}^{-1}\left(\frac{a_u^{\frac{8}{5}}}{k^{\frac{3}{5}}}\right),\qquad  c^*=\beta_{\mathcal{K}}^{-1}\left(\frac{a_l^{\frac{8}{5}}}{\left(a_uc_u\right)^{\frac{3}{5}}}\right), \]
to get the desired estimates for $c$. \\
(2) {\bf Estimates for $a$}: We recall \eqref{a,c2}. Then from $a_l\leq N\leq a_u$ and estimates of $c$ established above, we find
%\[
%a=\left({\int_{\mathbb{R}^3}\frac{1}{e^{|p|^2+c}\pm1}dp}\right)^\frac{2}{3}N^{-\frac{2}{3}}.
%\]
%We use $f\in\Lambda$, then we have
\begin{align*}
\left({\int_{\mathbb{R}^3}\frac{1}{e^{|p|^2+c^*}\pm1}dp}\right)^\frac{2}{3}a_u^{-\frac{2}{3}} \leq a\leq \left({\int_{\mathbb{R}^3}\frac{1}{e^{|p|^2+c_*}\pm1}dp}\right)^\frac{2}{3}a_l^{-\frac{2}{3}}.
\end{align*}
%Now, we define l.h.s as $a_*$ and r.h.s as $a^*$.
For boson case, $c_*\geq 0 $ implies the positivity of $a_*$. For fermion case, positivity of $a_*$ is trivial. This completes the proof.
%$a_* =\left({\int_{\mathbb{R}^3}\frac{1}{e^{|p|^2+c^*}-1}dp}\right)^\frac{2}{3}a_u^{-\frac{2}{3}} $ and $a^*= \left({\int_{\mathbb{R}^3}\frac{1}{e^{|p|^2+c_*}-1}dp}\right)^\frac{2}{3}a_l^{-\frac{2}{3}} $
\end{proof}
%%%%%%%%%%%%%%%%%%%%%%%%%%%%%%%%%%%%%%%%%%%%%%%%%%%%%%%%%%%%%%%%%%%%%%%%%%%%%%%%%%%%%%%%%%%%%%%%%%%%%%%%%%%%%%%%%%%%%%%%%%%%%%%%%%%%%%%%%%%%%%%%%%%%%%%%%%%%%%%%%%%
\subsection{Solution operator}
By Lemma \ref{determination}, the following solution operator $\Phi$ is well-defined on $\Lambda$:
\begin{definition}
We defind our solution operator $\Phi$ as
\begin{align*}
\Phi(f)=\Phi^+(f)1_{p_1>0}+\Phi^-(f)1_{p_1<0},
\end{align*}
where
\begin{align}\label{Phif+}
\begin{split}
\Phi^+(f)(x,p)&=e^{-\frac{1}{\tau|p_1|}\int_0^xN_f(y)dy}f_L(p)	\cr
&+\frac{1}{\tau|p_1|}\int_0^xe^{-\frac{1}{\tau|p_1|}\int_y^xN_f(z)dz}N_f(y)\mathcal{K}(f)dy \quad \qquad \textit{if} \quad p_1>0,
\end{split}
\end{align}
and
\begin{align}\label{Phif-}
\begin{split}
\Phi^-(f)(x,p)&=e^{-\frac{1}{\tau|p_1|}\int_x^1N_f(y)dy}f_R(p)	\cr
&+\frac{1}{\tau|p_1|}\int_x^1e^{-\frac{1}{\tau|p_1|}\int_x^yN_f(z)dz}N_f(y)\mathcal{K}(f)dy \quad \qquad \textit{if} \quad p_1<0.
\end{split}
\end{align}
\end{definition}
%assumption $(\mathcal{C})$ is needed to $c(x)\geq0$, which is necessary to $\mathcal{K}(f)$ has no condensation and no saturation.
%\end{definition}
%We prove the Theorem \ref{mainthm} by showing $\Phi(\Lambda) \subset \Lambda$ in Section $5$, and that $\Phi$ is a contraction mapping in Section $6$.

In the remaining sections, we show that $\Phi$ has a unique fixed point in $\Lambda$ if $\tau$ is sufficiently large. We first prove several estimates on
the quantum local equilibrium.
\begin{lemma}\label{Bf}
Let $f\in \Lambda$, then there exists a constant $C_{l,u}$ depending only on the quantities in (\ref{iniconst}) and $k$ such that
\[\mathcal{K}(f)(1+|p|^2) \leq C_{l,u}e^{-\frac{a_*}{4}|p|^2}.\]
\end{lemma}
\begin{proof}
We only consider $\mathcal{B}(f)|p|^2$. By an explicit computation, we have
%To obtain the boundedness of $\mathcal{B}(f)|p|^2$, we use $(a^2-b^2)=(a+b)(a-b)$ in numerator part, then we have
%First, we will obtain some boundedness of following two function $f$ and $g$:
%\begin{align*}
%f(x)&=\frac{1}{e^{ax^2+c}-1},\quad g(x)=\frac{x^2}{e^{ax^2+c}-1}.
%\end{align*}
%We differentiate $f$, then
%\begin{align*}
%f'(x)&=\frac{-2axe^{ax^2+c}}{\left(e^{ax^2+c}-1\right)^2}.
%\end{align*}
%The only crtical point is $x=0$, and in that point $f$ has global maximum $f(0)=\frac{1}{e^{c}-1}$. Similarly we differentiate $g$, then
%\[g'(x)=\frac{\left(e^{ax^2+c}-1\right)2x-2axe^{ax^2+c}x^2}{\left(e^{ax^2+c}-1\right)^2}	
%=\frac{2x\left(e^{ax^2+c}(1-ax^2)-1\right)}{\left(e^{ax^2+c}-1\right)^2}.\]
%$g$ has three critical points which are $0$ and two symmetric points $s$, $-s$. When $x=\pm s$, $g$ has global maximum $g(s)=g(-s)=\frac{s^2}{e^{as^2+c}-1}$ and when $x=0$, $g$ has local minimum $g(0)=0$. Therefore $f$ and $g$ are bounded. Now we can estimates $\mathcal{K}(f)|p|^2$.
\begin{align*}
\mathcal{B}(f)|p|^2 &= \frac{|p|^2}{e^{a\big|p-\frac{P}{N}\big|^2+c}-1}\cr
&\leq  \frac{|p|^2}{\left(e^{\frac{a_*}{2}\big|p-\frac{P}{N}\big|^2+\frac{c_*}{2}}-1\right)\left(e^{\frac{a_*}{2}\big|p-\frac{P}{N}\big|^2+\frac{c_*}{2}}+1\right)}	\cr
&\leq  \frac{2\big|p-\frac{P}{N}\big|^2+2\big|\frac{P}{N}\big|^2}{e^{\frac{a_*}{2}\big|p-\frac{P}{N}\big|^2+\frac{c_*}{2}}-1}\frac{1}{e^{\frac{a_*}{2}\big|p-\frac{P}{N}\big|^2+\frac{c_*}{2}}+1}.
\end{align*}
In last line, we used $a^2\leq 2|a-b|^2+2b^2$. Then, we observe
\begin{align}\label{PN}
\bigg|\frac{P}{N}\bigg| \leq \frac{a_u+c_u}{a_l},
\end{align}
which follows from $|P| \leq a_u+c_u$, and use the boundedness of $\frac{x^2+1}{e^{ax^2+c}-1}$ to get
\begin{align*}
\mathcal{B}(f)|p|^2&\leq  \frac{C_{a_*,c_*}}{e^{\frac{a_*}{2}\big|p-\frac{P}{N}\big|^2+\frac{c_*}{2}}+1}.	
\end{align*}
Now, since $|a-b|^2\geq a^2/2-b^2$, we have
\begin{align*}
\mathcal{B}(f)|p|^2&\leq  \frac{C_{l,u}}{e^{\frac{a_*}{2}\left(\frac{|p|^2}{2}-\big|\frac{P}{N}\big|^2\right)+\frac{c_*}{2}}+1}
=\frac{C_{l,u}}{e^{-\frac{a_*}{2}\big|\frac{P}{N}\big|^2}e^{\frac{a_*|p|^2}{4}+\frac{c_*}{2}}+1},
\end{align*}
%We then observe from the condition of ($\mathcal{A}$) of the $\Lambda$ and
We then use (\ref{PN}) again to get the desired result:
\begin{align*}
\mathcal{B}(f)|p|^2&\leq  C_{l,u} e^{\frac{a_*}{2}\big|\frac{P}{N}\big|^2}e^{-\frac{a_*|p|^2}{4}}e^{-\frac{c_*}{2}}	
\leq C_{l,u}e^{-\frac{a_*}{4}|p|^2}.
\end{align*}
\end{proof}
The following decay estimates are crucially used throughout the paper. The proof can be found in \cite{BY}. We provide detailed proof for reader's convenience.
\begin{lemma}\label{same form} We have
\begin{align*}
\int_0^x\int_0^{\infty}\frac{1}{\tau|p_1|}e^{-\frac{a_l(x-y)}{\tau|p_1|}}e^{-C_{l,u}p_1^2}dp_1dy \leq C_{l,u}\left(\frac{\ln\tau+1}{\tau}\right).
\end{align*}
\end{lemma}
\begin{proof}
We divide the integral domain of $p_1$ into three parts: %$0<p_1< \frac{1}{\tau}$, $\frac{1}{\tau}<p_1<\tau$ and $\tau<p_1$.
\begin{align*}
A&=\left\{\int_0^x\int_{0<p_1<\frac{1}{\tau}}+\int_0^x\int_{\frac{1}{\tau}<p_1<\tau}+\int_0^x\int_{p_1>\tau}\right\}\frac{1}{\tau|p_1|}e^{-\frac{a_l(x-y)}{\tau|p_1|}}e^{-C_{l,u}p_1^2}dp_1dy	\cr
&\equiv  I+II+III.
\end{align*}
Integrating in $y$ first, we get
\begin{align*}
I&=\int_{0<p_1<\frac{1}{\tau}}\int_0^x\frac{1}{\tau|p_1|}e^{-\frac{a_l(x-y)}{\tau|p_1|}}e^{-C_{l,u}p_1^2}dydp_1 \cr
&= \frac{1}{a_l} \int_{0<p_1<\frac{1}{\tau}}\left\{1-e^{-\frac{a_lx}{\tau|p_1|}}\right\}e^{-C_{l,u}p_1^2}dp_1	\cr
&\leq \frac{1}{a_l}\frac{1}{\tau}.
\end{align*}
%where we used $1-e^{-\frac{a_lx}{\tau|p_1|}}\leq 1$ and $ e^{-C_{l,u}p_1^2}\leq 1$.
We start similarly for $II$:
\begin{align*}
II&=\int_{\frac{1}{\tau}<p_1<\tau}\int_0^x\frac{1}{\tau|p_1|}e^{-\frac{a_l(x-y)}{\tau|p_1|}}e^{-C_{l,u}p_1^2}dydp_1	\cr
&\leq \frac{1}{a_l} \int_{\frac{1}{\tau}<p_1<\tau}\left\{1-e^{-\frac{a_lx}{\tau|p_1|}}\right\}e^{-C_{l,u}p_1^2}dp_1.
\end{align*}
Then we expand $e^{-\frac{a_lx}{\tau|p_1|}}$ in  Taylor expansion to obtain
\begin{align*}%\label{taylor}
\begin{split}
II&\leq \frac{1}{a_l} \int_{\frac{1}{\tau}<p_1<\tau}\left\{\left(\frac{a_l}{\tau|p_1|}\right)-\frac{1}{2!}\left(\frac{a_l}{\tau|p_1|}\right)^2+\frac{1}{3!}\left(\frac{a_l}{\tau|p_1|}\right)^3+\cdots \right\}dp_1	\cr
&\leq \frac{1}{a_l} \bigg|\int_{\frac{1}{\tau}}^{\tau}\left(\frac{a_l}{\tau p_1}\right)dp_1\bigg| +\frac{1}{a_l}\bigg|\int_{\frac{1}{\tau}}^{\tau}\frac{1}{2!}\left(\frac{a_l}{\tau p_1}\right)^2dp_1\bigg|+\frac{1}{a_l}\bigg|\int_{\frac{1}{\tau}}^{\tau}\frac{1}{3!}\left(\frac{a_l}{\tau p_1}\right)^3dp_1\bigg|+\cdots.
\end{split}
\end{align*}
Then, since
\begin{align*}
\int_{\frac{1}{\tau}}^{\tau}\left(\frac{1}{p_1}\right)^ndp= \left[\frac{1}{n-1}\frac{-1}{p_1^{n-1}}\right]_{\frac{1}{\tau}}^{\tau} = \frac{1}{n-1}\frac{r^{2n-2}-1}{r^{n-1}}.
\end{align*}
We can bound $II$ by
\begin{align*}
&\frac{1}{\tau}\ln\tau^2+\frac{1}{2!}\frac{a_l}{\tau^2}\frac{\tau^2-1}{\tau}+\frac{1}{2\cdot3!}\frac{a_l^2}{\tau^3}\frac{\tau^4-1}{\tau^2}+\frac{1}{3\cdot4!}\frac{a_l^3}{\tau^4}\frac{\tau^6-1}{\tau^3}+\cdots	\cr
%&= \frac{1}{\tau}\ln\tau^2+\frac{1}{2!}\frac{a_l}{\tau}\frac{\tau^2-1}{\tau^2}+\frac{1}{2\cdot3!}\frac{a_l^2}{\tau}\frac{\tau^4-1}{\tau^4}+\frac{1}{3\cdot4!}\frac{a_l^3}{\tau}\frac{\tau^6-1}{\tau^6}+\cdots	\cr
&\qquad\leq \frac{1}{\tau}\ln\tau^2+\frac{1}{2!}\frac{a_l}{\tau}+\frac{1}{3!}\frac{a_l^2}{\tau}+\frac{1}{4!}\frac{a_l^3}{\tau}+\cdots	\cr
&\qquad=\frac{1}{\tau}\ln\tau^2+\frac{e^{a_l}}{a_l}\frac{1}{\tau},
\end{align*}
where we used $(\tau^n-1)/\tau^n\leq 1$ in second line. Finally, by using $e^{-\frac{a_l(x-y)}{\tau|p_1|}}<1$, we estimate $III$ as
\begin{align*}
III&=\int_{p_1>\tau}\left\{\int_0^x\frac{1}{\tau|p_1|}e^{-\frac{a_l(x-y)}{\tau|p_1|}}dy\right\}e^{-C_{l,u}p_1^2}dp_1	
%&\leq \int_{p_1>\tau}\left\{\int_0^1\frac{1}{\tau|p_1|}dy\right\}e^{-C_{l,u}p_1^2}dp_1	\cr
%&\leq \frac{1}{\tau^2}\int_{p_1>\tau}e^{-C_{l,u}p_1^2}dp_1	\cr
\leq \frac{1}{\tau^2}\int_{\mathbb{R}}e^{-C_{l,u}p_1^2}dp_1	\leq C_{l,u}\frac{1}{\tau^2}.
\end{align*}
Combining the above estimates gives the desired results for sufficiently large $\tau$:
\begin{align*}
I+II+III \leq C_{l,u}\left\{\frac{1}{\tau}+\frac{1}{\tau}\ln\tau^2+\frac{1}{\tau^2}\right\} \leq C_{l,u}\left(\frac{\ln\tau+1}{\tau}\right).
\end{align*}
\end{proof}
%%%%%%%%%%%%%%%%%%%%%%%%%%%%%%%%%%%%%%%%%%%%%%%%%%%%%%%%%%%%%%%%%%%%%%%%%%%%%%%%%%%%%%%%%%%%%%%%%%%%%%%%%%%%%%%%%%%%%%%%%%%%%%%%%%%%%%%%%%%%%%%%%%%%%%%%%%%%%%%%%%%%%%%%%%%%%%
%
%
%
%
%
%%%%%%%%%%%%%%%%%%%%%%%%%%%%%%%%%%%%%%%%%%%%%%%%%%%%%%%%%%%%%%%%%%%%%%%%%%%%%%%%%%%%%%%%%%%%%%%%%%%%%%%%%%%%%%%%%%%%%%%%%%%%%%%%%%%%%%%%%%%%%%%%%%%%%%%%%%%%%%%%%%%%%%%%%%%%%%%55
\section{$\Phi$ maps $\Lambda$ into $\Lambda$}
The main result of this section is stated in the following proposition.
\begin{proposition}
Let $f_{LR}$ satisfies the assumptions in Theorem \ref{mainthm}. Then, there exists $\tau_0$ such that if $\tau>\tau_0$, then the solution operator $\Phi$ maps $\Lambda$ into $\Lambda$.
\end{proposition}
\begin{proof}
The proof is given in the following Lemma \ref{positive},  \ref{posi},  \ref{upperbdd} and \ref{lowerk}.
\end{proof}
\begin{lemma}\label{positive} Let $f\in \Lambda$. Assume $f_{LR}$ satisfies all the assumptions of the Theorem \ref{mainthm}. Then $\Phi(f)$ satisfies the following estimates:
\begin{align*}
\Phi(f)(x,p) \geq 0.
\end{align*}
\end{lemma}
\begin{proof}
%Clearly
%\begin{align*	\mathcal{F}(f) &= \frac{1}{e^{a\big|p-\frac{P}{N}\big|^2+c}+1}\geq 0.\end{align*}
Thanks to Lemma \ref{bddac}, the local equilibrium is strictly positive:
\begin{align*}
\mathcal{K}(f) &= \frac{1}{e^{a\big|p-\frac{P}{N}\big|^2+c}\pm1}
\geq \frac{1}{e^{a^*\big|p-\frac{P}{N}\big|^2+c^*}\pm1}	
>\frac{1}{e^{c^*}\pm1}
> 0.
\end{align*}
Therefore, we have from (\ref{Phif+}) and (\ref{Phif-}) that
\begin{align}\label{lowerbddphi}
\begin{split}
\Phi^+(f)(x,p)&\geq e^{-\frac{1}{\tau|p_1|}\int_0^xN_f(y)dy}f_L(p)\geq 0, \qquad \textit{if} \quad p_1>0, \cr
\Phi^-(f)(x,p)&\geq e^{-\frac{1}{\tau|p_1|}\int_x^1N_f(y)dy}f_R(p)\geq 0, \qquad \textit{if} \quad p_1<0,
\end{split}
\end{align}
%We can estimates $\Phi^-(f)$ by exactly same arguments from Lemma \ref{bddac} and positivity of $f_R$.
%\begin{align*}
%\Phi^-(f)(x,p) \geq 0.
%\end{align*}
which gives desired result.
\end{proof}
%Now we show that $\Phi(f)$ satisfies the condition $(\mathcal{B})$ of the $\Lambda$ in Lemma 4.2, 4.3 and 4.5.
\begin{lemma}\label{posi} Let $f\in \Lambda$. Assme $f_{LR}$ satisfies all the assumptions of the Theorem \ref{mainthm}, then $\Phi(f)$ also satisfies the following inequality.
\begin{align*}
\int_{\mathbb{R}^3}\Phi(f)dp \geq a_l,	\qquad  \int_{\mathbb{R}^3}\Phi(f)|p|^2dp \geq c_l.
\end{align*}
\end{lemma}
\begin{proof}
We only consider the second one. We see from \eqref{lowerbddphi} that
\begin{align}\label{philower}
\begin{split}
\Phi(f)\geq e^{-\frac{1}{\tau|p_1|}\int_0^xN_f(y)dy}f_L(p)1_{p_1>0}+e^{-\frac{1}{\tau|p_1|}\int_x^1N_f(y)dy}f_R(p)1_{p_1<0}.
\end{split}
\end{align}
Using $N_f \leq a_u$, we see that
\begin{align*}
\Phi(f) &\geq e^{-\frac{a_ux}{\tau|p_1|}}f_L(p)1_{p_1>0}+e^{-\frac{(1-x)a_u}{\tau|p_1|}}f_R(p)1_{p_1<0}	\cr
&\geq e^{-\frac{a_u}{\tau|p_1|}}f_L(p)1_{p_1>0}+e^{-\frac{a_u}{\tau|p_1|}}f_R(p)1_{p_1<0}	\cr
&= e^{-\frac{a_u}{\tau|p_1|}}f_{LR}.
\end{align*}
We then integrate with respect to $|p|^2dp$ to get the desired results:
%\begin{align*}
%\Phi(f)|p|^2 &\geq e^{-\frac{a_u}{\tau|p_1|}}f_{LR}|p|^2.
%\end{align*}
%We now take $\int_{\mathbb{R}^3}dp$ in the above two results, and by definition of $a_l$ and $c_l$ in (\ref{iniconst}), we obtain the desired results.
%\begin{align*}
%\int_{\mathbb{R}^3}\Phi(f)dp \geq	%\int_{\mathbb{R}^3}e^{-\frac{a_u}{\tau|p_1|}}f_{LR} dp = a_l,
%\end{align*}
\begin{align*}
\int_{\mathbb{R}^3}\Phi(f)|p|^2dp \geq	\int_{\mathbb{R}^3}e^{-\frac{a_u}{\tau|p_1|}}f_{LR}|p|^2dp = c_l.
\end{align*}
\end{proof}

\begin{lemma}\label{upperbdd} Let $f\in \Lambda$. Assume $f_{LR}$ satisfies all the assumptions of the Theorem \ref{mainthm}. Then $\Phi(f)$ satisfies the following estimates:
\begin{align*}
\int_{\mathbb{R}^3}\Phi(f)dp \leq a_u,	\qquad \int_{\mathbb{R}^3}\Phi(f)|p|^2dp \leq c_u,
\end{align*}
for sufficiently large $\tau$.
\end{lemma}
\begin{proof}
We only consider the second inequality. We integrate \eqref{Phif+} with respect to $|p|^2dp$  to get
\begin{align}\label{upperesti}
\begin{split}
\int_{\mathbb{R}^3}\Phi^+(f)|p|^2dp&=\int_{p_1>0}e^{-\frac{1}{\tau|p_1|}\int_0^xN_f(y)dy}f_L(p)|p|^2dp	\cr
&+\int_{p_1>0}\frac{1}{\tau|p_1|}\int_0^xe^{-\frac{1}{\tau|p_1|}\int_y^xN_f(z)dz}N_f(y)\mathcal{K}(f)|p|^2dydp.
\end{split}
\end{align}
Since $N_f\geq a_l$ and $x\geq0$, we can estimate first term as
\begin{align*}
\int_{p_1>0}e^{-\frac{1}{\tau|p_1|}\int_0^xN_f(y)dy}f_L(p)|p|^2dp \leq
\int_{p_1>0}e^{-\frac{xa_l}{\tau|p_1|}}f_L(p)|p|^2dp
\leq \int_{p_1>0}f_L(p)|p|^2dp.
\end{align*}
We then recall Lemma \ref{Bf} and use $a_l\leq N_f \leq a_u$ to bound the second term as
\begin{align*}
&\int_{p_1>0}\frac{1}{\tau|p_1|}\int_0^xe^{-\frac{1}{\tau|p_1|}\int_y^xN_f(z)dz}N_f(y)\mathcal{K}(f)|p|^2dydp	\cr
&\quad \leq  C_{l,u }\int_0^x\int_0^{\infty}\frac{1}{\tau|p_1|}e^{-\frac{1}{\tau|p_1|}\int_y^xN_f(z)dz}N_f(y)e^{-C_{l,u}p_1^2}dydp_1\int_{\mathbb{R}}e^{-C_{l,u}\{p_2^2+p^2_3\}}dp_2dp_3.\cr
&\quad \leq  C_{l,u }a_u\int_0^x\int_0^{\infty}\frac{1}{\tau|p_1|}e^{-\frac{a_l(x-y)}{\tau|p_1|}}e^{-C_{l,u}p_1^2}dydp_1.
\end{align*}
Therefore, we have from Lemma \ref{same form} that
\begin{align*}
\int_{p_1>0}\Phi^+(f)|p|^2dp& \leq \int_{p_1>0}f_L(p)|p|^2dp	+C_{l,u}\left(\frac{\ln\tau+1}{\tau}\right).
\end{align*}
%If we multiply $|p|^2$ in (\ref{Phif-}) and integrate with respect to $\int_{\mathbb{R}^3}dp$ then $\Phi^-(f)|p|^2$ part also can be estimated similarly as follows:
%Estimates of $\int_{p_1<0}\Phi^-(f)|p|^2$ part can be obtained in a similar manner:
Similarly, we can derive
\begin{align*}
\int_{p_1<0}\Phi^-(f)|p|^2dp& \leq \int_{p_1<0}f_R(p)|p|^2dp	+C_{l,u}\left(\frac{\ln\tau+1}{\tau}\right),
\end{align*}
so that
\begin{align*}
\int_{\mathbb{R}^3}\Phi(f)|p|^2dp& \leq \frac{1}{2}c_u +C_{l,u}\left(\frac{\ln\tau+1}{\tau}\right).
\end{align*}
which gives the desired result for sufficiently large $\tau$.
\end{proof}
%Now we show that $\Phi(f)$ satisfies condition $(\mathcal{C})$ of $\Lambda$ when $f\in\Lambda$ in following two lemmas.
\begin{lemma}\label{LemP} Let $f\in \Lambda$. Assume $f_{LR}$ satisfies all the assumptions of the Theorem \ref{mainthm}. Then, for sufficiently large $\tau$, we have
\begin{align*}
\bigg|\int_{\mathbb{R}^3}\Phi(f)p_idp\bigg| \leq C_{l,u}\left(\frac{\ln \tau+1}{\tau}\right),
\end{align*}
for $i=2,3$.
\end{lemma}
\begin{proof}
We only consider the case $i=2$. For this, we integrate (\ref{Phif+}) with respect to $p_2dp_2dp_3$:
\begin{align*}%\label{P+}
\begin{split}
\int_{\mathbb{R}^2}\Phi^+(f)(x,p)p_2dp_2dp_3&=\int_{\mathbb{R}^2}e^{-\frac{1}{\tau|p_1|}\int_0^xN_f(y)dy}f_L(p)p_2dp_2dp_3	\cr
&+\int_{\mathbb{R}^2}\frac{1}{\tau|p_1|}\int_0^xe^{-\frac{1}{\tau|p_1|}\int_y^xN_f(z)dz}N_f(y)\mathcal{K}(f)p_2dydp_2dp_3.
\end{split}
\end{align*}
We note that the first term in r.h.s vanishes due to the assumption (3) of Theorem \ref{mainthm}:
\begin{align*}
\int_{\mathbb{R}^2}e^{-\frac{1}{\tau|p_1|}\int_0^xN_f(y)dy}f_L(p)p_2dp_2dp_3 =e^{-\frac{1}{\tau|p_1|}\int_0^xN_f(y)dy}\int_{\mathbb{R}^2}f_L(p)p_2dp_2dp_3=0.
\end{align*}
%The second term of \eqref{P+} is estimated as\begin{align*}
%\bigg|\int_{\mathbb{R}^2}\mathcal{K}(f)p_2dp_2dp_3\bigg| \leq
%\bigg|\int_{\mathbb{R}^2}\mathcal{K}(f)(1+|p|^2)dp_2dp_3\bigg|,\end{align*}
%which, combined with Lemma \ref{Bf}, gives
%\begin{align*}
%\bigg|\int_{\mathbb{R}^2}\mathcal{K}(f)p_2dp_2dp_3\bigg| \leq \int_{\mathbb{R}^2}C_{l,u}e^{-C_{l,u}|p|^2}dp_2dp_3 \leq C_{l,u}e^{-C_{l,u}|p_1|^2}.
%\end{align*}
For the second term, we use $a_l \leq N_f\leq a_u$ and employ Lemma \ref{Bf} to derive
\begin{align*}
\bigg|\int_{\mathbb{R}^2}\Phi^+(f)(x,p)p_2dp_2dp_3\bigg|&\leq C_{l,u}\frac{1}{\tau|p_1|}\int_0^xe^{-\frac{1}{\tau|p_1|}\int_y^xN_f(z)dz}N_f(y)e^{-C_{l,u}|p_1|^2}dy	\cr
&\leq  C_{l,u}\frac{a_u}{\tau|p_1|}\int_0^xe^{-\frac{a_l(x-y)}{\tau|p_1|}}e^{-C_{l,u}|p_1|^2}dy.
\end{align*}
Now we integrate with respect to $ dp_1$ on $p_1>0$  to obtain
\begin{align*}	
\bigg|\int_{p_1>0}\Phi^+(f)(x,p)p_2dp\bigg| &\leq \int_{p_1>0}\bigg|\int_{\mathbb{R}^2}\Phi^+(f)(x,p)p_2dp_2dp_3\bigg|dp_1\cr
&\leq  C_{l,u}\int_{p_1>0}\frac{1}{\tau|p_1|}\int_0^xe^{-\frac{a_l(x-y)}{\tau|p_1|}}e^{-C_{l,u}|p_1|^2}dydp_1.
\end{align*}
Therefore, we have from Lemma \ref{same form} that
\begin{align*}	
\bigg|\int_{p_1>0}\Phi^+(f)(x,p)p_2dp\bigg| &\leq  C_{l,u}\left(\frac{\ln\tau+1}{\tau}\right).
\end{align*}
Similarly, we have
\begin{align*}
\bigg|\int_{p_1<0}\Phi^-(f)(x,p)p_2dp\bigg| &\leq  C_{l,u}\left(\frac{\ln\tau+1}{\tau}\right),
\end{align*}
which gives the desired result.
\end{proof}

\begin{lemma}\label{lowerk} Let $f\in \Lambda$. Assume $f_{LR}$ satisfies all the assumptions of the Theorem \ref{mainthm}. Then, for sufficiently large $\tau$, we have
\begin{align*}
\left(\int_{\mathbb{R}^3}\Phi(f)dp\right)\left(\int_{\mathbb{R}^3}\Phi(f)|p|^2dp\right)-\bigg|\int_{\mathbb{R}^3}\Phi(f)pdp\bigg|^2 \geq k.
\end{align*}
\end{lemma}
\begin{proof}
We have from Cauchy-Schwarz inequality that
\begin{align}\label{CS}
\begin{split}
&\int_{\mathbb{R}^3}\Phi(f)dp\int_{\mathbb{R}^3}\Phi(f)|p|^2dp-\left|\int_{\mathbb{R}^3}\Phi(f)pdp\right|^2\cr
&\hspace{1.8cm}\geq \left(\int_{\mathbb{R}^3}|p|\Phi(f)dp\right)^2-\left|\int_{\mathbb{R}^3}p\Phi(f)dp\right|^2 \cr
&\hspace{1.8cm}\geq \left(\int_{\mathbb{R}^3}|p_1|\Phi(f)dp\right)^2-\left|\int_{\mathbb{R}^3}p\Phi(f)dp\right|^2 \cr
&\hspace{1.8cm}=\left(\int_{\mathbb{R}^3}|p_1|\Phi(f)dp\right)^2-\left|\int_{\mathbb{R}^3}p_1\Phi(f)dp\right|^2-R,
\end{split}
\end{align}
where
\begin{align*}
R&=M_2^2+M_3^2+2M_1M_2+2M_2M_3+2M_1M_3,
\end{align*}
and
\begin{align*}
M_i&=\int_{\mathbb{R}^3}p_i\Phi(f)dp \quad \textit{for} \quad i=1,2,3.
\end{align*}
From Lemma \ref{upperbdd}, we can bound $M_1$ as
\begin{align*}
M_1 \leq \bigg| \int_{\mathbb{R}^3}(1+|p|^2)\Phi(f)dp \bigg| \leq a_u+c_u,
\end{align*}
%\[M_1\leq |P| \leq a_u+c_u,\]
and $M_2$, $M_3$ decay as Lemma \ref{LemP}
\begin{align*}
M_2, ~M_3\leq C_{l,u}\left(\frac{\ln\tau+1}{\tau}\right).
\end{align*}
Therefore,
\begin{align}\label{R}
R\leq C_{l,u}\left(\frac{\ln\tau+1}{\tau}\right).
\end{align}
On the other hand, we use $a^2-b^2=(a+b)(a-b)$ to get
\begin{align}\label{kesti}
\begin{split}
&\left(\int_{\mathbb{R}^3}|p_1|\Phi(f)dp\right)^2-\left|\int_{\mathbb{R}^3}p_1\Phi(f)dp\right|^2\cr
 &\hspace{2cm}\geq \left\{\int_{\mathbb{R}^3}\big(|p_1|+p_1\big)\Phi(f)dp\right\}\left\{\int_{\mathbb{R}^3}\big(|p_1|-p_1\big)\Phi(f)dp\right\} \cr
&\hspace{2cm}=4\int_{p_1>0}p_1\Phi(f)dp\int_{p_1<0}|p_1|\Phi(f)dp,
\end{split}
\end{align}
%Recall the lower bound of $\Phi$ in \eqref{philower}. Integrating with respect to $p_1dp$ for $p_1>0$ on each sides of (\ref{philower}) gives
and observe that from the definition of $\Phi$, and property $(\mathcal{B})$ of $\Lambda$: $N_f \leq a_u$  that
\begin{align*}
\int_{p_1>0}p_1\Phi(f)dp\geq \int_{p_1>0}p_1e^{-\frac{1}{\tau|p_1|}\int_0^xN_f(y)dy}f_L(p)dp\geq \int_{p_1>0}p_1e^{-\frac{a_u}{\tau|p_1|}}f_L(p)dp
\end{align*}
and%Similarly we take $\int_{p_1<0}|p_1|dp$ on each sides of \eqref{philower} to get
\begin{align*}
\int_{p_1<0}|p_1|\Phi(f)dp\geq \int_{p_1<0}|p_1|e^{-\frac{1}{\tau|p_1|}\int_x^1N_f(y)dy}f_R(p)dp\geq \int_{p_1<0}|p_1|e^{-\frac{a_u}{\tau|p_1|}}f_R(p)dp.
\end{align*}
We insert these lower bounds into (\ref{kesti}) and recall the definition of $k$ in (\ref{k}) to obtain
%from Inserting these two inequaltiy into \eqref{kesti} and using $N_f \leq a_u$ yield
\begin{align}\label{kk}
\left(\int_{\mathbb{R}^3}|p_1|\Phi(f)dp\right)^2-\left|\int_{\mathbb{R}^3}p_1\Phi(f)dp\right|^2\geq 4k.
\end{align}
From (\ref{CS}), (\ref{R}) and (\ref{kk}), we have
\begin{align*}
\left(\int_{\mathbb{R}^3}\Phi(f)dp\right)\left(\int_{\mathbb{R}^3}\Phi(f)|p|^2dp\right)-\bigg|\int_{\mathbb{R}^3}\Phi(f)pdp\bigg|^2 \geq 4k-C_{l,u}\left(\frac{\ln\tau+1}{\tau}\right),
\end{align*}
which, for sufficiently large $\tau$, gives the desired result.
\end{proof}
\section{Continuity of quantum equilibrium $\mathcal{K}$}
In this section, we establish the continuity property of the quantum equilibrium $\mathcal{K}$, which is crucially used to show the contractiveness of $\Phi$ in Section 5.
\subsection{Transitional quantum local equilibrium $\mathcal{K}(\theta)$}
In this subsection, we define a transitional quantum local equilibrium. We start with the convexity  of our solution space.
\begin{lemma}\label{convex}
Let $f,g\in \Lambda$, Then the linear combination $(1-\theta)f+\theta g$ lies in $\Lambda$ for $\theta \in [0,1]$.
\end{lemma}
\begin{proof}
Since the conditions $(\mathcal{A})$ and $(\mathcal{B})$ of $\Lambda$ are trivially satisfied, we only consider $(\mathcal{C})$. For this, we define a functional $G$ by
\begin{align*}
G(f)=\left(\int_{\mathbb{R}^3}f(x,p)dp\right)\left(\int_{\mathbb{R}^3}f(x,p)|p|^2dp\right)-\left(\int_{\mathbb{R}^3}f(x,p)pdp\right)^2,
\end{align*}
and a matrix $M$ by
\begin{align*}
M(f)=\left(
\begin{array}{cc}
N_{f}&P_{f}\cr
P_{f}&E_{f}
\end{array}
\right),
\end{align*}
for $f\in \Lambda$.
%\begin{align*}G(\theta)=\left(\int_{\mathbb{R}^3}f(x,p)dp\right)\left(\int_{\mathbb{R}^3}f(x,p)|p|^2dp\right)-\left(\int_{\mathbb{R}^3}f(x,p)pdp\right)^2.\end{align*}
We note that
\[
G(f)=\det M(f).
\]
Then, by Brum-Minkowski inequality, we have for $f,g\in \Lambda$
\begin{align*}
G(\theta f+(1-\theta)g)&=\det M(\theta f+(1-\theta)g)\cr
&\geq\left\{\det M(f)\right\}^{\theta}\left\{\det M(g)\right\}^{1-\theta}\cr
&\geq\left\{G(f)\right\}^{\theta}\left\{G(g)\right\}^{1-\theta}\cr
&\geq k^{\theta}k^{1-\theta}\cr
&=k.
\end{align*}
Therefore, $\theta f+(1-\theta)g\in \Lambda$.
\end{proof}
We now define the transitional macroscopic fields constructed from the linear combination $\theta f+(1-\theta)g$ as
\[
(N_{\theta},P_{\theta},E_{\theta})= (1-\theta) (N_{f},P_{f},E_{f})+ \theta(N_{g},P_{g},E_{g}),
\]
for $\theta \in [0,1]$. Now, since we have shown in Lemma \ref{convex} that $\theta f+(1-\theta)g\in \Lambda$, the  existence of the unique quantum equilibrium $\mathcal{K}(\theta)$:
\begin{align*}
\mathcal{K}(\theta)=\frac{1}{e^{a_{\theta}(x)\big|p-\frac{P(x)}{N(x)}\big|^2+c_{\theta}(x)}\pm1}
\end{align*}
which shares the same mass, momentum and energy with  $\theta f+(1-\theta)g$:
%Before doing so, we should check that the parameter $a_{\theta}$ and $c_{\theta}$ can be well defined.
%We want to construct a transitional quantum equilibrium, that is, a Fermi-Dirac distribution (in the fermion case), or a Bose-Einstein distribution (in the boson case) $\mathcal{K}(\theta)$ that is determined by
\begin{align*}
\int_{\mathbb{R}^3}\mathcal{K}(\theta)dp=N_{\theta},\quad \int_{\mathbb{R}^3}\mathcal{K}(\theta)dp=P_{\theta}\quad \int_{\mathbb{R}^3}\mathcal{K}(\theta)dp=E_{\theta}
\end{align*}
is guaranteed by Lemma \ref{determination}. We also recall from  Lemma \ref{bddac} that $a_{\theta}$ and $c_{\theta}$ are determined by
\begin{align}\label{actheta}
c_{\theta}=\beta_{\mathcal{K}}^{-1}\left(\frac{N_{\theta}}{\left(E_{\theta}-\frac{P_{\theta}^2}{N_{\theta}}\right)^{\frac{3}{5}}}\right), \quad a_{\theta}=\left(\frac{\int_{\mathbb{R}^3}\frac{1}{e^{|p|^2+c_{\theta}}\pm 1}dp}{N_{\theta}}\right)^\frac{2}{3},
\end{align}
and satisfy
\begin{align*}
a_*\leq a_{\theta}\leq a^*,\qquad c_*\leq c_{\theta}\leq c^*.
\end{align*}
for some positive constants $a_*$, $a^*$, $c^*$ and $c_*$.
%\begin{lemma}
%This remark guarantees the well definedness of $c_{\theta}$ and $a_{\theta}$ which are denoted in \eqref{actheta}. Since $(1-\theta)f+\theta g \in \Lambda$, all Lemmas in Section 3 are still satisfied for $N_{\theta},P_{\theta},E_{\theta}$ and $a_{\theta},c_{\theta}$. Therefore, we can estimates these parameters.
%\end{lemma}

%Before we prove the remaining part, we state some remarks which are useful during the proof.
%We state some remarks before proceeding with further proofs.

\subsection{Derivatives of $\mathcal{F}(\theta)$}  We now derive  derivative estimates of $a_{\theta}$ and $c_{\theta}$, which will be needed later in the proof of the
continuity estimate of $\mathcal{K}(\theta)$. We first need the following estimate of $\beta_{\mathcal{K}}$.
\begin{lemma}\label{betaupper} Let $f,g\in\Lambda$, then $\beta_{\mathcal{K}}$ defined in \eqref{beta} satisfies
\begin{align*}
\bigg|\frac{1}{\beta_{\mathcal{K}}'(c_{\theta})}\bigg|< C_{l,u},
\end{align*}
where $C_{l,u}$ depends on constants of (\ref{iniconst}) and $k$.
\end{lemma}
\begin{proof}
By definition given in (\ref{beta}), $\beta_{\mathcal{K}}$ is an infinitely differentiable function.
On the other hand, Lemma \ref{monotone} implies that $\beta^{\prime}_{\mathcal{K}}(c)< 0$. Therefore, we see from Lemma \ref{bddac}
that $\beta^{\prime}_{\mathcal{K}}(c)$ is a strictly negative continuous function defined on a closed interval $[c_*,c^*]$. Hence, there exists positive $C$ such that $|\beta_{\mathcal{K}}'(c)|\geq C$, which gives the desired result. %where $C$ depends on $c_*$ and $c^*$, and which depend on constants defined in \eqref{iniconst} and $k$.
%we observe that $\beta_{\mathcal{K}}(c)$ is continuously differentiable function in closed interval. Furthermore by Lemma \ref{monotone}, we can see that there exists $C$ such that $|\beta_{\mathcal{K}}'(c)|\geq C$ where $C$ depends on $c_*$ and $c^*$, and which depend on constants defined in \eqref{iniconst} and $k$.
\end{proof}
%The remaining proof of Proposition \ref{banaBf} is resolved in next two Lemmas. We want to compute $\frac{\partial \mathcal{K}(\theta)}{\partial N_{\theta}}$,$\frac{\partial \mathcal{K}(\theta)}{\partial P_{\theta}}$ and $\frac{\partial \mathcal{K}(\theta)}{\partial E_{\theta}}$. But $\mathcal{K}(\theta)$ is function of $a_{\theta}$ and $c_{\theta}$. So we first estimates derivatives of $a_{\theta}$ and $c_{\theta}$ with respect to transitional macroscopic fields $N_{\theta}$, $P_{\theta}$ and $E_{\theta}$.
\begin{lemma}\label{diffc} We have
\begin{align*}
\bigg|\left(\frac{\partial c_{\theta}}{\partial N_\theta},\frac{\partial c_\theta}{\partial P_\theta},\frac{\partial c_\theta}{\partial E_\theta}\right) \bigg| \leq C_{l,u}.
\end{align*}
\end{lemma}
\begin{proof}
Recall that $c_{\theta}$ is function of $N_{\theta}$, $P_{\theta}$ and $E_{\theta}$:
\begin{align*}
c_{\theta}=\beta_{\mathcal{K}}^{-1}\left\{ \frac{N_{\theta}}{(E_{\theta}-\frac{P_{\theta}^2}{N_{\theta}})^{\frac{3}{5}}}\right\}.
\end{align*}
(1) By an explicit computation, we get
\begin{align*}
\bigg|\frac{\partial c_\theta}{\partial N_\theta}\bigg|
= \bigg|\frac{1}{\beta_{\mathcal{K}}'(c_\theta)}\bigg| \bigg|\frac{\partial}{\partial N_\theta}\left( \frac{N_\theta}{\left(E_\theta-\frac{P_\theta^2}{N_\theta}\right)^{\frac{3}{5}}}\right)	\bigg|
%&= \frac{1}{\beta_{\mathcal{K}}'(c_\theta)} \frac{\left(E_\theta-\frac{P_\theta^2}{N_\theta}\right)^{\frac{3}{5}}-\frac{3}{5}\left(E_\theta-\frac{P_\theta^2}{N_\theta}\right)^{-\frac{2}{5}}\frac{P_\theta^2}{N_\theta^2}N_\theta}{\left(E_\theta-\frac{P_\theta^2}{N_\theta}\right)^{\frac{6}{5}}} 	\cr
%&= \frac{1}{\beta_{\mathcal{K}}'(c_\theta)}  \frac{\left(E_\theta-\frac{P_\theta^2}{N_\theta}\right)-\frac{3}{5}\frac{P_\theta^2}{N_\theta}}{\left(E_\theta-\frac{P_\theta^2}{N_\theta}\right)^{\frac{8}{5}}} 	\cr
= \bigg|\frac{1}{\beta_{\mathcal{K}}'(c_\theta)}\bigg| \bigg|\frac{E_\theta-\frac{8}{5}\frac{P_\theta^2}{N_\theta}}{\left(E_\theta-\frac{P_\theta^2}{N_\theta}\right)^{\frac{8}{5}}}\bigg|.
\end{align*}
We then use $N_{\theta}\leq a_u$, $E_{\theta}\leq c_u$ and $N_{\theta}E_{\theta}-P_{\theta}^2\geq k$ together with Lemma \ref{betaupper} to obtain
\begin{align*}
\bigg|\frac{\partial c_\theta}{\partial N_\theta}\bigg|\leq  \bigg|\frac{1}{\beta_{\mathcal{K}}'(c_\theta)}\bigg| \frac{a_u^{\frac{8}{5}}c_u}{k^{\frac{8}{5}}}	\leq  C_{l,u}.
\end{align*}
(2) Similarly, we compute
\begin{align*}
\bigg|\frac{\partial c_\theta}{\partial P_\theta}\bigg|
= \bigg|\frac{1}{\beta_{\mathcal{K}}'(c_\theta)}\bigg| \bigg|\frac{\partial}{\partial P_\theta}\left(\ \frac{N_\theta}{(E_\theta-\frac{P_\theta^2}{N_\theta})^{\frac{3}{5}}}\right)	\bigg|
%&= \frac{1}{\beta_{\mathcal{K}}'(c_\theta)} \frac{-N_\theta\frac{3}{5}(E_\theta-\frac{P_\theta^2}{N_\theta})^{-\frac{2}{5}}(-\frac{2P_\theta}{N_\theta})}{(E_\theta-\frac{P_\theta^2}{N_\theta})^{\frac{6}{5}}}	\cr
= \bigg|\frac{1}{\beta_{\mathcal{K}}'(c_\theta)}\bigg| \frac{\frac{6}{5}|P_\theta|}{(E_\theta-\frac{P_\theta^2}{N_\theta})^{\frac{8}{5}}}.
\end{align*}
Since  $|P_{\theta}| \leq a_u+c_u$, we have
\begin{align*}
\bigg|\frac{\partial c_\theta}{\partial P_\theta}\bigg|
\leq \bigg|\frac{1}{\beta_{\mathcal{K}}'(c_\theta)}\bigg| \frac{\frac{6}{5}(a_u+c_u){a_u}^{\frac{8}{5}}}{k^{\frac{8}{5}}}	
\leq  C_{l,u}.
\end{align*}
(3) In an almost identical manner, we compute
\begin{align*}
\bigg|\frac{\partial c_\theta}{\partial E_\theta}\bigg|
&= \bigg|\frac{1}{\beta_{\mathcal{K}}'(c_\theta)}\bigg| \bigg|\frac{\partial}{\partial E_\theta}\left( \frac{N_\theta}{(E_\theta-\frac{P_\theta^2}{N_\theta})^{\frac{3}{5}}}\right)	\bigg|\cr
%&= \frac{1}{\beta_{\mathcal{K}}'(c_\theta)}  \frac{-N_\theta\frac{3}{5}(E_\theta-\frac{P_\theta^2}{N_\theta})^{-\frac{2}{5}}}{(E_\theta-\frac{P_\theta^2}{N_\theta})^{\frac{6}{5}}}	\cr
&= \bigg|\frac{1}{\beta_{\mathcal{K}}'(c_\theta)} \bigg| \frac{\frac{3}{5}|N_\theta|}{(E_\theta-\frac{P_\theta^2}{N_\theta})^{\frac{8}{5}}}\cr	
&\leq \bigg|\frac{1}{\beta_{\mathcal{K}}'(c_\theta)} \bigg| \frac{\frac{3}{5}{a_u}^{\frac{13}{5}}}{k^{\frac{8}{5}}}	\cr
&\leq  C_{l,u}.
\end{align*}
\end{proof}
%Now, wcompute differential of $a_{\theta}$ with respect to macroscopic quantities since $a_{\theta}$ depends on $c_{\theta}$.
\begin{lemma}\label{diffa} We have%If we differentiate $a_{\theta}$ with respect to macroscopic quantities, then we have
\begin{align*}
\bigg|\left(\frac{\partial a_\theta}{\partial N_\theta}, \frac{\partial a_\theta}{\partial P_\theta}, \frac{\partial a_\theta}{\partial E_\theta} \right)\bigg|\leq C_{l,u}.
\end{align*}
\end{lemma}
\begin{proof}
%Recall that $a_{\theta}$ is a function of $N_{\theta}$ and $c_{\theta}$.
%\begin{align*}a_{\theta}&=\left(\frac{\int_{\mathbb{R}^3}\frac{1}{e^{|p|^2+c_{\theta}}\pm1}dp}{N_{\theta}}\right)^\frac{2}{3}.\end{align*}
(1) We recall (\ref{actheta}) and compute
\begin{align*}
\frac{\partial a_\theta}{\partial N_\theta}
&= \frac{2}{3} \left(\frac{\int_{\mathbb{R}^3}\frac{1}{e^{|p|^2+c_\theta}\pm1}dp}{N_\theta}\right)^{-\frac{1}{3}}
\frac{\partial}{\partial N_\theta } \left(\frac{\int_{\mathbb{R}^3}\frac{1}{e^{|p|^2+c_\theta}\pm1}dp}{N_\theta}\right)	\cr
&= \frac{2}{3} \left(\frac{\int_{\mathbb{R}^3}\frac{1}{e^{|p|^2+c_\theta}\pm1}dp}{N_\theta}\right)^{-\frac{1}{3}}
\left(\frac{N_\theta \int_{\mathbb{R}^3}\frac{-e^{|p|^2+c_\theta}}{(e^{|p|^2+c_\theta}\pm1)^2}\frac{\partial c_\theta}{\partial N_\theta}dp-\int_{\mathbb{R}^3}\frac{1}{e^{|p|^2+c_\theta}\pm1}dp}{N_\theta^2}\right).
\end{align*}
It then follows directly from from $a_l\leq N_{\theta}\leq a_u$, Lemma \ref{bddac}, and Lemma \ref{diffc} that
\begin{align*}
\frac{\partial a_\theta}{\partial N_\theta}
&\leq C \left(\int_{\mathbb{R}^3}\frac{1}{e^{|p|^2+c^*}\pm1}dp\right)^{-\frac{1}{3}}a_u^{\frac{1}{3}}
\left(\frac{C_{l,u}a_u\int_{\mathbb{R}^3}\frac{1}{e^{|p|^2+c_*}\pm1}dp+\int_{\mathbb{R}^3}\frac{1}{e^{|p|^2+c_*}\pm1}dp}{a_l^2}\right)	\cr
&\leq C_{l,u}.
\end{align*}
(2) In a similar manner, we have
\begin{align*}
\left(\frac{\partial a_\theta}{\partial P_\theta}\right)_i
&= \frac{2}{3} \left(\frac{\int_{\mathbb{R}^3}\frac{1}{e^{|p|^2+c_\theta}\pm1}dp}{N_\theta}\right)^{-\frac{1}{3}}
\frac{\partial}{\partial P_{\theta i} } \left(\frac{\int_{\mathbb{R}^3}\frac{1}{e^{|p|^2+c_\theta}\pm1}dp}{N_\theta}\right)	\cr
&= \frac{2}{3} \left(\frac{\int_{\mathbb{R}^3}\frac{1}{e^{|p|^2+c_\theta}\pm1}dp}{N_\theta}\right)^{-\frac{1}{3}}
\left(\frac{\int_{\mathbb{R}^3}\frac{-e^{|p|^2+c_\theta}}{(e^{|p|^2+c_\theta}\pm1)^2}\frac{\partial c_\theta}{\partial P_{\theta i}}dp}{N_\theta}\right)	\cr
&\leq C \left(\int_{\mathbb{R}^3}\frac{1}{e^{|p|^2+c_*}\pm1}dp\right)^{-\frac{1}{3}}a_u^{\frac{1}{3}}
\left(C_{l,u}\frac{\int_{\mathbb{R}^3}\frac{1}{e^{|p|^2+c_*}\pm1}dp}{a_l}\right)	\cr
&\leq C_{l,u}.
\end{align*}
(3) Replacing $\frac{\partial }{\partial P_{\theta i}}$ by $\frac{\partial}{\partial E_\theta}$ in (2), we get the same result for $\frac{\partial a_\theta}{\partial E_\theta}$.
%(3) Analogously, we obtain
%\begin{align*}
%\frac{\partial a_\theta}{\partial E_\theta}
%&= \frac{2}{3} \left(\frac{\int_{\mathbb{R}^3}\frac{1}{e^{|p|^2+c_\theta}\pm1}dp}{N_\theta}\right)^{-\frac{1}{3}}
%\frac{\partial}{\partial E_\theta} \left(\frac{\int_{\mathbb{R}^3}\frac{1}{e^{|p|^2+c_\theta}\pm1}dp}{N_\theta}\right)	\cr
%&= \frac{2}{3} \left(\frac{\int_{\mathbb{R}^3}\frac{1}{e^{|p|^2+c_\theta}\pm1}dp}{N_\theta}\right)^{-\frac{1}{3}}
%\left(\frac{\int_{\mathbb{R}^3}\frac{-e^{|p|^2+c_\theta}}{(e^{|p|^2+c_\theta}\pm1)^2}\frac{\partial c_\theta}{\partial E_\theta}dp}{N_\theta}\right)	\cr
%&\leq C \left(\int_{\mathbb{R}^3}\frac{1}{e^{|p|^2+c^*}\pm1}dp\right)^{-\frac{1}{3}}a_u^{\frac{1}{3}}\left(C_{l,u}\frac{\int_{\mathbb{R}^3}\frac{1}{e^{|p|^2+c_*}\pm1}dp}{a_l}\right)	\cr
%&\leq C_{l,u}.
%\end{align*}
\end{proof}
\subsection{Continuity of $\mathcal{K}$}
We now prove the main result of the this section:
\begin{proposition}\label{banaBf} Let $f,g \in \Lambda$. Then the quantum equilibrium $\mathcal{K}$ satisfies following property: %Bose-Einstein distribution $\mathcal{B}$ and Fermi-Dirac distribution $\mathcal{F}$ satisfy following property.
\begin{align*}
|\mathcal{K}(f)-\mathcal{K}(g)| \leq C_{l,u} \sup_{x}||f-g||_{L_2^1}e^{-C_{l,u}|p|^2}.
\end{align*}
\end{proposition}
\begin{proof}

%Because $(1-\theta) f +\theta g$ also is in $\Lambda$,  $\mathcal{B}(\theta)$ has no condensation and $\mathcal{F}(\theta)$ is not saturated distribution.
We apply taylor's theorem around $\theta=0$ to have
\begin{align*}
\mathcal{K}(1)-\mathcal{K}(0)=\int_0^1 \mathcal{K}'(\theta) d\theta,
\end{align*}
so that
\begin{align}\label{bfbg}
\begin{split}
\mathcal{K}(f)-\mathcal{K}(g) &= (N_g-N_f)\int_0^1 \frac{\partial \mathcal{K}(\theta)}{\partial N_{\theta}} d\theta	\cr
&+ (P_g-P_f)\int_0^1 \frac{\partial \mathcal{K}(\theta)}{\partial P_{\theta}} d\theta	\cr
&+ (E_g-E_f)\int_0^1 \frac{\partial \mathcal{K}(\theta)}{\partial E_{\theta}} d\theta.
\end{split}
\end{align}
% To compute \eqref{bfbg}, we differentiate quantum equilibrium with respect to macroscopic fields by using above two lemmas.\\
To estimate the first integral, we compute
\begin{align*}
\frac{\partial \mathcal{K}(\theta)}{\partial N_{\theta}}&= \frac{- \left\{ \frac{\partial a_{\theta}}{\partial N_{\theta}}|p-\frac{P_{\theta}}{N_{\theta}}|^2+a_{\theta}\frac{2P_{\theta}}{N_{\theta}^2}(p-\frac{P_{\theta}}{N_{\theta}})+\frac{\partial c_{\theta}}{\partial N_{\theta}} \right\}e^{a_{\theta}|p-\frac{P_{\theta}}{N_{\theta}}|^2+c_{\theta}}}{(e^{a_{\theta}|p-\frac{P_{\theta}}{N_{\theta}}|^2+c_{\theta}}\pm1)^2}.
\end{align*}
From Lemma \ref{Bf} we observe $\mathcal{K}(f) \leq C_{l,u}$ to obtain
\begin{align*}
\frac{e^{a_{\theta}|p-\frac{P_{\theta}}{N_{\theta}}|^2+c_{\theta}}}{e^{a_{\theta}|p-\frac{P_{\theta}}{N_{\theta}}|^2+c_{\theta}}\pm1} = 1\mp \frac{1}{e^{a_{\theta}|p-\frac{P_{\theta}}{N_{\theta}}|^2+c_{\theta}}\pm1} \leq C_{l,u}.
\end{align*}
With these computation and Lemma \ref{bddac}, Lemma \ref{diffc} and Lemma \ref{diffa}, we get
%and $\big|\frac{P_{\theta}}{N_{\theta}}\big|\leq C_{l,u}$
\begin{align*}
\bigg|\frac{\partial \mathcal{K}(\theta)}{\partial N_{\theta}}\bigg|&\leq  C_{l,u}\left(\bigg|p-\frac{P_{\theta}}{N_{\theta}}\bigg|^2+\bigg|p-\frac{P_{\theta}}{N_{\theta}}\bigg|+1\right)\frac{1}{e^{a_{\theta}|p-\frac{P_{\theta}}{N_{\theta}}|^2+c_{\theta}}\pm1}.
\end{align*}
Since  $|P| \leq a_u+c_u$ and $N_{\theta}\geq a_{\ell}$, we find
\begin{align*}
\bigg|\frac{\partial \mathcal{K}(\theta)}{\partial N_{\theta}}\bigg| &\leq  C_{l,u}\left(|p|^2+1\right)\mathcal{K}(\theta),
\end{align*}
which, thanks to Lemma \ref{Bf}, gives
\begin{align*}
\bigg|\frac{\partial \mathcal{K}(\theta)}{\partial N_{\theta}}\bigg| &\leq C_{l,u}e^{-C_{l,u}|p|^2}.
\end{align*}
 %and
%\begin{align*}
%\frac{e^{a_{\theta}|p-\frac{P_{\theta}}{N_{\theta}}|^2+c_{\theta}}}{(e^{a_{\theta}|p-\frac{P_{\theta}}{N_{\theta}}|^2+c_{\theta}}\pm1)^2} \leq \mathcal{K}(\theta).
%\end{align*}
Similarly, we have $(i=1,2,3)$
\begin{align*}
\bigg|\left(\frac{\partial \mathcal{K}(\theta)}{\partial P_{\theta}}\right)_i\bigg|&=\bigg|\frac{-\left\{ \frac{\partial a_{\theta}}{\partial P_{\theta}}|p-\frac{P_{\theta}}{N_{\theta}}|^2-a_{\theta}\frac{2}{N_{\theta}}(p-\frac{P_{\theta}}{N_{\theta}})+\frac{\partial c_{\theta}}{\partial P_{\theta}} \right\}e^{a_{\theta}|p-\frac{P_{\theta}}{N_{\theta}}|^2+c_{\theta}}}{(e^{a_{\theta}|p-\frac{P_{\theta}}{N_{\theta}}|^2+c_{\theta}}\pm1)^2}\bigg|	\cr
&\leq  C_{l,u}\left(|p|^2+1\right)\frac{1}{e^{a_{\theta}|p-\frac{P_{\theta}}{N_{\theta}}|^2+c_{\theta}}\pm1}	\cr
%&\leq  C_{l,u}\left(|p|^2+1\right)\frac{1}{e^{a_{\theta}|p-\frac{P_{\theta}}{N_{\theta}}|^2+c_{\theta}}\pm1}	\cr
&\leq C_{l,u}e^{-C_{l,u}|p|^2},
\end{align*}
and
\begin{align*}
\bigg|\frac{\partial \mathcal{K}(\theta)}{\partial E_{\theta}}\bigg| &= \bigg|\frac{-\left\{ \frac{\partial a_{\theta}}{\partial E_{\theta}}|p-\frac{P_{\theta}}{N_{\theta}}|^2+\frac{\partial c_{\theta}}{\partial E_{\theta}} \right\}e^{a_{\theta}|p-\frac{P_{\theta}}{N_{\theta}}|^2+c_{\theta}}}{(e^{a_{\theta}|p-\frac{P_{\theta}}{N_{\theta}}|^2+c_{\theta}}\pm1)^2}\bigg|	\cr
&\leq  C_{l,u}\left(|p|^2+1\right)\frac{1}{e^{a_{\theta}|p-\frac{P_{\theta}}{N_{\theta}}|^2+c_{\theta}}\pm1}	\cr
%&\leq  C_{l,u}\left(|p|^2+1\right)\frac{1}{e^{a_{\theta}|p-\frac{P_{\theta}}{N_{\theta}}|^2+c_{\theta}}\pm1}	\cr
&\leq C_{l,u}e^{-C_{l,u}|p|^2}.
\end{align*}
Substituting these estimates into (\ref{bfbg}) yields the desired result:
\begin{align*}
&|\mathcal{K}(f)-\mathcal{K}(g)| \cr
%&\quad \leq |N_f-N_g|\int_0^1 \bigg|\frac{\partial \mathcal{K}(\theta)}{\partial N_{\theta}}\bigg| d\theta + |P_f-P_g|\int_0^1 \bigg|\frac{\partial \mathcal{K}(\theta)}{\partial P_{\theta}} \bigg|d\theta+ |E_f-E_g|\int_0^1 \bigg|\frac{\partial \mathcal{K}(\theta)}{\partial E_{\theta}}\bigg| d\theta	\cr
&\quad \leq \left(\bigg|\int_{\mathbb{R}^3}(f-g)dp\bigg| + \bigg|\int_{\mathbb{R}^3}(f-g)pdp\bigg| + \bigg|\int_{\mathbb{R}^3}(f-g)|p|^2dp\bigg|\right) C_{l,u}e^{-C_{l,u}|p|^2}\cr
&\quad \leq C_{l,u}\sup_{x}||f-g||_{L_2^1}e^{-C_{l,u}|p|^2}.
\end{align*}
\end{proof}
\section{$\Phi$ is contractive in $\Lambda$} It remains to show that $\Phi$ is a contraction mapping in $\Lambda$ for sufficiently large $\tau$.
\begin{proposition}Let $f,g \in \Lambda$ and $f_{LR}$ satisfies all the assumptions of the Theorem \ref{mainthm}, then, for sufficietly large $\tau$, $\Phi$ satisfies
\begin{align*}
\sup_{x\in[0,1]}||\Phi(f)-\Phi(g)||_{L_2^1} \leq \alpha \sup_{x\in[0,1]}||f-g||_{L_2^1},
\end{align*}
for some constant $0<\alpha<1$.
\end{proposition}
\begin{proof}
%Consider $\Phi^+(f)$ in (\ref{Phif+}) and define $I(f)$ and $II(f,g,h)$ as follows:
We only estimate $\Phi^+$. Let
\begin{align*}
\Phi^+(f)=I(f)+II(f,f,f),
\end{align*}
where
\begin{align*}
I(f)=e^{-\frac{1}{\tau|p_1|}\int_0^xN_f(y)dy}f_L(p),
\end{align*}
and
\begin{align*}
II(f,g,h)=\frac{1}{\tau|p_1|}\int_0^xe^{-\frac{1}{\tau|p_1|}\int_y^xN_f(z)dz}N_g(y)\mathcal{K}(h)dy.
\end{align*}
\noindent$\bullet$ {\bf Estimates for $I(f)-I(g)$}:
Consider
\begin{align*}
I(f)-I(g)=\left\{e^{-\frac{1}{\tau|p_1|}\int_0^xN_f(y)dy}-e^{-\frac{1}{\tau|p_1|}\int_0^xN_g(y)dy}\right\}f_L(p),
\end{align*}
which, by mean value theorem, can be rewritten as
\begin{align}\label{MVT}
\begin{split}
e^{-\frac{1}{\tau|p_1|}\int_0^xN_f(y)dy}&-e^{-\frac{1}{\tau|p_1|}\int_0^xN_g(y)dy}	\cr
&= -\frac{1}{\tau|p_1|}e^{-\frac{1}{\tau|p_1|}\int_0^x(1-\mu)N_f(y)+\mu N_g(y)dy}\int_0^xN_f(y)-N_g(y)dy,
\end{split}
\end{align}
for some $0<\mu<1$.
%from the mean value theorem on $j(\theta)$:
%\begin{align*}
%j(\theta)=e^{-\frac{1}{\tau|p_1|}\int_0^x(1-\theta)N_f(y)+\theta N_g(y)dy}.
%\end{align*}
Since we have $N_f,N_g\geq a_l$, we see that
\begin{align*}
|I(f)-I(g)|%&=\bigg|e^{-\frac{1}{\tau|p_1|}\int_0^xN_f(y)dy}-e^{-\frac{1}{\tau|p_1|}\int_0^xN_g(y)dy}\bigg|f_L(p)	\cr
&\leq \frac{1}{\tau|p_1|}e^{-\frac{1}{\tau|p_1|}\int_0^x(1-\theta)N_f(y)+\theta N_g(y)dy}\int_0^x|N_f(y)-N_g(y)|dyf_L(p)	\cr
%&\leq \frac{1}{\tau|p_1|}e^{-\frac{1}{\tau|p_1|}\int_0^x(1-\theta)a_l+\theta a_ldy}\sup_{x\in[0,1]}||f-g||_{L^1_2}f_L(p)	\cr
&\leq \frac{1}{\tau|p_1|}e^{-\frac{-a_lx}{\tau|p_1|}}\sup_{x\in[0,1]}||f-g||_{L^1_2}f_L(p),
\end{align*}
where we used
\begin{align*}
	|N_f(y)-N_g(y)|\leq \sup_{x\in[0,1]}||f-g||_{L^1_2}.
\end{align*}
Now we integrate each term with respect to $(1+|p|^2)dp$ on $p_1>0$:
\begin{align*}
\int_{p_1>0}|I(f)-I(g)|(1+|p|^2)dp	&\leq \int_{p_1>0}\frac{1}{\tau|p_1|}e^{-\frac{-a_lx}{\tau|p_1|}}f_L(p)(1+|p|^2)dp\sup_{x\in[0,1]}||f-g||_{L^1_2}	\cr
&\leq \frac{1}{\tau}\int_{p_1>0}\frac{1}{|p_1|}f_L(p)(1+|p|^2)dp\sup_{x\in[0,1]}||f-g||_{L^1_2}	\cr
&\leq \frac{1}{\tau}(a_s+c_s)\sup_{x\in[0,1]}||f-g||_{L^1_2},
\end{align*}
to get the desired result.\newline
%and Taking supremum on each sides gives
%\begin{align*}
%\sup_{x\in[0,1]}||I(f)-I(g)||_{L^1_2} \leq \frac{1}{\tau}(a_s+c_s)\sup_{x\in[0,1]}||f-g||_{L^1_2}.
%\end{align*}
\noindent$\bullet$ {\bf Estimates for $II(f)-II(g)$:}
We split it as
\begin{align*}
II(f,f,f)-II(g,g,g)&=\{II(f,f,f)-II(g,f,f)\}+\{II(g,f,f)-II(g,g,f)\}	\cr
& +\{II(g,g,f)-II(g,g,g)\} \cr
&= II_1+II_2+II_3.
\end{align*}
\noindent (i) Estimate of $II_1$: In a similar manner as in \eqref{MVT}, we get
%we can find $0<\mu<1$ such that
%\begin{align*}
%e^{-\frac{1}{\tau|p_1|}\int_y^xN_f(z)dz}&-e^{-\frac{1}{\tau|p_1|}\int_y^xN_g(z)dz}\cr
%&=-\frac{1}{\tau|p_1|}e^{-\frac{1}{\tau|p_1|}\int_y^x(1-\mu)N_f(z)+\mu N_g(z)dz}\int_y^xN_f(z)-N_g(z)dz.
%\end{align*}
%From $f,g\in\Lambda$, we use $N_f,N_g\geq a_l$, to obtain
\begin{align*}
\bigg|e^{-\frac{1}{\tau|p_1|}\int_y^xN_f(z)dz}-e^{-\frac{1}{\tau|p_1|}\int_y^xN_g(z)dz}\bigg|
&\leq \frac{x-y}{\tau|p_1|}e^{-\frac{a_l(x-y)}{\tau|p_1|}}\sup_{x\in[0,1]}||f-g||_{L^1_2}	\cr
&\leq \frac{C}{a_l}e^{-\frac{a_l(x-y)}{2\tau|p_1|}}\sup_{x\in[0,1]}||f-g||_{L^1_2}.
\end{align*}
In last line, we used $xe^{-x}\leq Ce^{-\frac{x}{2}}$. From this, we see that
\begin{align*}
\int_{p_1>0}&|II_1|(1+|p|^2)dp \cr
&\leq \int_{p_1>0}\frac{1}{\tau|p_1|}\int_0^x\bigg|e^{-\frac{1}{\tau|p_1|}\int_y^xN_f(z)dz}-e^{-\frac{1}{\tau|p_1|}\int_y^xN_g(z)dz}\bigg|N_f(y)\mathcal{K}(f)(1+|p|^2)dydp	\cr
&\leq  \frac{Ca_u}{a_l}\int_{p_1>0}\frac{1}{\tau|p_1|}\int_0^xe^{-\frac{a_l(x-y)}{2\tau|p_1|}}\mathcal{K}(f)(1+|p|^2)dydp\sup_{x\in[0,1]}||f-g||_{L^1_2}.
\end{align*}
We then apply Lemma \ref{Bf}
\begin{align*}
\int_{p_1>0}|II_1|(1+|p|^2)dp
&\leq  C_{l,u}\int_{p_1>0}\frac{1}{\tau|p_1|}\int_0^xe^{-\frac{a_l(x-y)}{2\tau|p_1|}}e^{-C_{l,u}|p|^2}dydp\sup_{x\in[0,1]}||f-g||_{L^1_2},
\end{align*}
and Lemma \ref{same form} to obtaini
\begin{align*}
\int_{p_1>0}|II_1|(1+|p|^2)dp
&\leq C_{l,u}\left(\frac{\ln\tau+1}{\tau}\right)\sup_{x\in[0,1]}||f-g||_{L^1_2}.
\end{align*}
\noindent (i) Estimate of $II_2$: The estimate for $II_2$ is treated similarly:
\begin{align*}
\int_{p_1>0}&|II_2|(1+|p|^2)dp \cr
&\leq \int_{p_1>0}\frac{1}{\tau|p_1|}\int_0^xe^{-\frac{1}{\tau|p_1|}\int_y^xN_g(z)dz}|N_f(y)-N_g(y)|\mathcal{K}(f)(1+|p|^2)dydp	\cr
&\leq \int_{p_1>0}\frac{1}{\tau|p_1|}\int_0^xe^{-\frac{a_l(x-y)}{\tau|p_1|}}C_{l,u}e^{-C_{l,u}|p|^2}dydp||f-g||_{L^1_2}	\cr
&\leq  C_{l,u}\left(\frac{\ln\tau+1}{\tau}\right)\sup_{x\in[0,1]}||f-g||_{L^1_2}.
\end{align*}
%where we used Lemma \ref{same form} again.
\noindent (iii) Estimate of $II_3$: we integrate with respect to $(1+|p|^2)dp$ on $p_1>0$ to obtain
\begin{align*}
\int_{p_1>0}&|II_3|(1+|p|^2)dp \cr
&\leq  \int_{p_1>0}\frac{1}{\tau|p_1|}\int_0^xe^{-\frac{1}{\tau|p_1|}\int_y^xN_g(z)dz}N_g(y)|\mathcal{K}(f)-\mathcal{K}(g)|(1+|p|^2)dydp.
\end{align*}
We then apply the continuity property of $\mathcal{K}$ in Proposition \ref{banaBf}:
\begin{align*}
\int_{p_1>0}&|II_3|(1+|p|^2)dp \cr
&\leq a_u\int_{p_1>0}\frac{1}{\tau|p_1|}\int_0^xe^{-\frac{a_l(x-y)}{\tau|p_1|}}C_{l,u}e^{-C_{l,u}|p|^2}(1+|p|^2)dydp\sup_{x\in[0,1]}||f-g||_{L^1_2}	\cr
&\leq C_{l,u}\left(\frac{\ln\tau+1}{\tau}\right)\sup_{x\in[0,1]}||f-g||_{L^1_2}.
\end{align*}
%In last line, we used Lemma \ref{same form} again.
Combining all these estimates, we get the desired estimate for $\Phi_+$:
\begin{align*}
\sup_{x\in[0,1]}||\Phi^+(f)-\Phi^+(g)||_{L_2^1} \leq C_{l,u}\left(\frac{1}{\tau}(a_s+c_s)+\frac{\ln\tau+1}{\tau}\right)\sup_{x\in[0,1]}||f-g||_{L^1_2}.
\end{align*}
The corresponding estimate for $\Phi_-$ can be derived in an identical manner:
\begin{align*}
\sup_{x\in[0,1]}||\Phi^-(f)-\Phi^-(g)||_{L_2^1} \leq C_{l,u}\left(\frac{1}{\tau}(a_s+c_s)+\frac{\ln\tau+1}{\tau}\right)\sup_{x\in[0,1]}||f-g||_{L^1_2}.
\end{align*}
This gives the desired contractive estimate for $\Phi$ when $\tau$ is sufficiently large.
%\begin{align*}
%\sup_{x\in[0,1]}||\Phi(f)-\Phi(g)||_{L_2^1} \leq C_{l,u}\left(\frac{1}{\tau}(a_s+c_s)+\frac{\ln\tau+1}{\tau}\right)\sup_{x\in[0,1]}||f-g||_{L^1_2}.
%\end{align*}
%For sufficiently large $\tau$, we derive the desired results.
\end{proof}

\noindent {\bf Acknowledgement:}
S.-B. Yun was supported by Basic Science  Research  Program  through  the  National  Research  Foundation  of  Korea  (NRF)  funded  by  the Ministry of Education (NRF-2016R1D1A1B03935955)
%%%%%%%%%%%%%%%%%%%%%%%%%%%%%%%%%%%%%%%%%%%%%%%%%%%%%%%%%%%%%%%%%%%%%%%%%%%%%%%%%%%%%%%%%%%%%%%%%%%%%%%%%%%%%%%%%%%%%%%%%%%%
%
\bibliographystyle{amsplain}

\end{document}